\documentclass[
,11pt%
,pagesize%
,headings=big%
,paper=a4%
,parskip=false%
,headsepline=true%
,abstracton
]{scrartcl}

\addtokomafont{sectioning}{\normalfont\bfseries}
\setkomafont{title}{\normalfont\LARGE}
\setkomafont{subtitle}{\normalfont\Large}
\addtokomafont{pageheadfoot}{\scshape\small}
\usepackage[draft]{fixme}
\usepackage{pdfsync} 
\usepackage[english]{babel} 
\usepackage[utf8]{inputenc}
\usepackage[T1]{fontenc} 
\usepackage[usenames,dvipsnames]{xcolor} 
\usepackage{rotfloat} 
\usepackage{pdfpages} 
\usepackage{graphicx} 
\usepackage[usenames,dvipsnames]{xcolor}
\usepackage[format=plain, indention=0.2cm]{caption}
\usepackage{url} 
\usepackage{booktabs} 
\usepackage{textcomp} 
\usepackage{marvosym} 
\usepackage{multirow} 
\usepackage{makecell} 
\usepackage[linesnumberedhidden,ruled,vlined]{algorithm2e}
\usepackage{amsfonts, amsmath, amsthm, amssymb, stmaryrd}
\usepackage{dsfont}
\usepackage{upgreek}
\usepackage{todonotes}
\usepackage[onehalfspacing]{setspace}
\usepackage{nicefrac}
\usepackage{txfonts}
\usepackage{bbold}
\usepackage{tikz-cd}
\usepackage[section]{placeins}

\newcommand{\ct}{u}
\newcommand{\st}{y}
\newcommand{\ad}{p}

\newcommand{\mix}{z}
\newcommand{\mixad}{q} 

\usepackage[final,%
pdftex,%
bookmarks,%
bookmarksdepth=3,%
breaklinks=true
]{hyperref}%

\newtheorem{assumption}{Assumption}
\theoremstyle{plain}

\newtheorem{theorem}[assumption]{Theorem}
\theoremstyle{plain}

\newtheorem{remark}[assumption]{Remark}
\theoremstyle{plain}

\theoremstyle{plain}

\newtheorem{lemma}[assumption]{Lemma}
\theoremstyle{plain}

\theoremstyle{plain}

\theoremstyle{plain}

\newtheorem{corollary}[assumption]{Corollary}
\theoremstyle{plain}

\usepackage{aliascnt}
\makeatletter
\let\c@algocf\relax 
\makeatother
\newaliascnt{algocf}{assumption}
\makeatother

\DeclareMathOperator{\supp}{supp}
\DeclareMathOperator{\Span}{span}

\DeclareMathOperator{\sign}{sign}

\definecolor{Horange}{RGB}{200,100,0}

\title{Variational discretization of one-dimensional elliptic optimal control problems with BV functions based on the mixed formulation}
\author{Evelyn Herberg\footnote{Department of Mathematical Sciences, George Mason University, 4400 University Dr, Fairfax, VA 22030, U.S.A.}$\,$ , Michael Hinze\footnote{Mathematisches Institut, Universität Koblenz-Landau, Campus Koblenz, Universitätsstraße 1, 56070 Koblenz, Germany.}}
\date{\today}

\begin{document}

	\maketitle
	
\textbf{Abstract.} We consider optimal control of an elliptic two-point boundary value problem governed by functions of bounded variation (BV). The cost functional is composed of a tracking term for the state and the BV-seminorm of the control. We use the mixed formulation for the state equation together with the variational discretization approach, where we use the classical lowest order Raviart-Thomas finite elements for the state equation. Consequently the variational discrete control is a piecewise constant function over the finite element grid. We prove error estimates for the variational discretization approach in combination with the mixed formulation of the state equation and confirm our analytical findings with numerical experiments. \\



\section{Problem formulation} \label{sec:HHNProblem}


We consider the optimal control problem
\begin{equation}
	\min_{\ct \in BV(\varOmega)} J(\ct) \coloneqq \tfrac{1}{2} \| \st - \st_{\operatorname{d}} \|_{L^2(\varOmega)}^2 + \alpha \, \| \ct' \|_{\mathcal{M}(\varOmega)},
	\tag{$P$}
	\label{eq:HHNproblem}
\end{equation}
where $\st$ satisfies the one-dimensional elliptic {equation
\begin{equation}
	\begin{cases}
		- (a \st')' + d \st &= \ct \qquad \text{in} \; \varOmega ,\\
		\qquad\qquad\;\, \st &= 0 \qquad \text{on} \; \Gamma. \\
	\end{cases}
	\label{eq:HHNPDE}
\end{equation}
}

Here $\varOmega =(0,1)$ with boundary $\Gamma = \left\{0,1\right\}$, and $\alpha >0$ is a given parameter. We assume $a \in W^{1,\infty}(\varOmega), a \ge a_0 > 0$ a.e. in $\varOmega$, where $a_0$ is a constant, and $d \in L^{\infty}(\varOmega), d \ge 0$ a.e. in $\varOmega$.
We denote the control by $\ct$ which we seek in $BV(\varOmega)$, the state by $\st \in H^1_0(\varOmega)$, and the desired state by $\st_{\operatorname{d}} \in L^{\infty}(\varOmega)$.

Our work is motivated by \cite{hafemeyer2019}, where a similar optimal control problem is considered. There, variational discretization (from \cite{VD}) combined with the classical piecewise linear and continuous finite element approximation of the state is investigated, and also a fully discrete approach with piecewise constant control approximations.
	The variational discrete approach leads to the approximation order of $h^2$ for both, control in $L^1$ and state in $L^2$, whereas the fully discrete approach only gives the optimal approximation order $h$.
We here propose a variational discrete approach, which automatically delivers piecewise constant control approximations.
{Although this approach also only delivers the optimal approximation order of $h$ for the state in $L^2$ and for the control in $L^1$, it allows a more elegant and more natural numerical analysis than the fully discrete approach presented in \cite{hafemeyer2019}.
This is achieved by formulating the elliptic partial differential equation in its mixed form. Variational discretization based on the classical Raviart-Thomas discretization of the state equation then delivers piecewise constant control approximations, while keeping the corresponding variationally discrete, reduced optimization problem infinite-dimensional. This in turn then simplifies the numerical analysis, since e.g. the optimal control of the continuous problem in this approach can be used as comparison function in the variationally discretized optimization problem.

We give a brief overview of related literature. An early result in optimization with BV functions and regularization by BV-seminorms is \cite{CasasKunischPola}. Further studies involving BV functions are \cite{Bartels2012,BartelsMilicevic,BrediesVicente}. There exist studies of elliptic optimal control with total variation regularization and control in $L^{\infty}(\varOmega)$, see \cite{ClasonKruseKunisch,HinzeKaltenbacherQuyen}. Controls from the space $BV(\varOmega) \cap L^{\infty}(\varOmega)$ are considered in \cite{CasasKogutLeugering}. 
Optimal control governed by a semilinear parabolic equation and control cost in a total bounded variation seminorm is discussed in \cite{CasasKruseKunisch}, a convergence result is shown and numerical experiments are presented. A similar problem is analyzed in \cite{CasasKunischBV}, but with semilinear elliptic equation. Numerical results for problems with BV-control are derived in \cite{trautmannwalter2021}.
In \cite{HinzeQuyen} the BV source in an elliptic system is recovered. 
Furthermore, we remark that the inherent sparsity structure of the problem is closely related to the sparsity structure observed in optimal control problems with measures control, see e.g. \cite{CasKun,CK19,HH,HerbergHS}.

We structure this work as follows: In Section~\ref{sec:HHNCont} we introduce the mixed formulation of the state equation, prove existence of a unique solution to the elliptic optimal control problem and derive its optimality conditions and sparsity structure. We apply variational discretization to the problem in Section~\ref{sec:HHNVD} and discuss the resulting structure of the non-discretized controls. Then, we proceed analogously to the analysis of the continuous problem by proving existence of a solution, deriving optimality conditions and sparsity structure. We also examine error estimates. Finally, in Section~\ref{sec:HHNCompRes} we explain the numerical implementation and present computational results for two different examples. We compare our findings to the experiments from \cite{hafemeyer2019}.

\section{Continuous optimality system} \label{sec:HHNCont}
We begin by examining the state equation. We set $\mix = a \st'$. Then the state $\st$ is supposed to solve \eqref{eq:HHNPDE} in the following weak sense: Find $(\mix,\st) \in H^1(\varOmega) \times L^2(\varOmega) $, such that
\begin{subequations}\label{eq:HHNPDEmixedweak}
	\begin{align}
		\int_{\varOmega} \left(\frac{1}{a} \mix v + v' \st \right) \, \textcolor{black}{\rm{dx}} &= 0 \qquad\quad\qquad\;\;\;\, \forall v \in H^1(\varOmega),\\
		\int_{\varOmega} \left( - \mix' w + d \st w \right) \, \textcolor{black}{\rm{dx}}  &=  \int_{\varOmega} w \ct \, \textcolor{black}{\rm{dx}} \qquad\;\;\, \forall w \in L^2(\varOmega) .
	\end{align}
\end{subequations} 
\noindent Here we write $(\mix,\st) = (\mix(\ct),\st(\ct))$ for the solution of \eqref{eq:HHNPDEmixedweak}. We know by \cite[Theorem 1]{RaviartThomas} that $(\mix(\ct),\st(\ct))$ admits a unique solution $(\mix,\st) \in H^1(\varOmega) \times H^1_0(\varOmega)$, where $\st$ solves \eqref{eq:HHNPDE} and $\mix = a \st'$. Furthermore, we define the forms $a(\mix,v):= \int_{\varOmega} \frac{1}{a} \mix v$, $b(v,\st):= \int_{ \varOmega} v' \st$, $c(\st,w):= \int_{ \varOmega} d \st w$ for all $(v,w) \in H^1(\varOmega)\times L^2(\varOmega)$. Then, for given $\ct \in L^2(\varOmega)$, the pair $(\mix,\st) = (\mix(\ct),\st(\ct))$ solves \eqref{eq:HHNPDEmixedweak}, iff
\begin{align*}
	& a(\mix,v) + b(v,\st) - b(\mix,w) + c(\st,w) = (\ct,w)\textcolor{black}{_{L^2(\varOmega)} =:}
	 \left(
	 \begin{pmatrix}
		0 \\ \ct
	\end{pmatrix},
\begin{pmatrix}
	v \\ w
\end{pmatrix}
 \right)\textcolor{black}{_\varOmega}
 	\quad \forall (v,w) \in H^1(\varOmega)\times L^2(\varOmega).
\end{align*}
Analogously, we for $s \in L^2(\varOmega)$ define the pair $(\mixad, \ad)= (\mixad(s),\ad(s))$ as the unique solution to
\begin{align} \label{eq:HHNadjointequation}
	a(v,\mixad)+b(\mixad,w)-b(v,\ad)+c(w,\ad) = (s,w)\textcolor{black}{_{L^2(\varOmega)} =:}
	 \left(
	\begin{pmatrix}
		0 \\ s
	\end{pmatrix} ,
	\begin{pmatrix}
	v \\ w
\end{pmatrix} 
\right)\textcolor{black}{_\varOmega}
	\quad \forall (v,w) \in H^1(\varOmega)\times L^2(\varOmega).
\end{align} 

\begin{remark}
	We note that we also may allow data in the first component of the vectors $(0,u)^{\top}, (0,s)^{\top}$. However, we in the present work only consider control problems which directly affect the state $\st$ by the control $\ct$, and only observations of the state $\st$, not of the derivative $\mix$ of $\st$.
	Introducing the mixed formulation offers the opportunity of including quantities containing $\mix=\st'$ in the target functional. 
	Furthermore, \eqref{eq:HHNadjointequation} constitutes the weak form of the adjoint equation associated to \eqref{eq:HHNPDE}.
\end{remark}

%
%
Let us note that problem \eqref{eq:HHNPDEmixedweak} for $\ct \in L^2(\varOmega)$ admits a unique solution $(\mix(\ct),\st(\ct)) \in H^1_0(\varOmega) \times H^1(\varOmega)$, which satisfies
	\begin{equation}\label{eq:HHNstateH2}
	||\st(\ct)||_{H^2(\varOmega)} + ||\mix(\ct)||_{H^1(\varOmega)} \leq C ||\ct ||_{L^2(\varOmega)},
\end{equation}
with some $C>0$, compare \cite[Lemma 2.2.]{gong2011mixed}.
%
%
%
%
%
Also, we with \cite[Theorem 2.2.]{hafemeyer2019} directly have

\begin{theorem}\label{thm:HHNuniquesolution}
	Problem \eqref{eq:HHNproblem} admits a unique solution $\bar \ct \in BV(\varOmega)$ with associated optimal state $\bar \st \in H^1_0(\varOmega) \cap H^2(\varOmega)$ and associated $\bar \mix \in H^1(\varOmega)$.
\end{theorem}

\noindent Similar to \cite[Theorem 2.3.]{hafemeyer2019}, but adapted to the mixed formulation of the state equation, we provide the following optimality conditions.

\begin{theorem}\label{thm:HHNoptcond}
	The control $\bar \ct \in BV(\varOmega)$ with associated $(\bar \mix,\bar\st) \in H^1(\varOmega) \times H^1_0(\varOmega) \cap H^2(\varOmega)$ is optimal for the problem \eqref{eq:HHNproblem} if and only if there exists a unique pair $(\bar \mixad,\bar \ad) \in H^1(\varOmega) \times H^1_0(\varOmega)\cap H^2(\varOmega) $, such that $(\bar \ct, \bar \mix, \bar \st, \bar \mixad, \bar \ad)$ and the $H^3(\varOmega)$ function $\bar \Phi (x) \coloneqq \int_0^x \bar \ad (s) \, ds $ satisfy $\bar \Phi(1) = 0$ as well as
	\begin{alignat}{2}
		\int_{ \varOmega} \bar \Phi \, d \bar \ct ' &= \alpha || \bar \ct ' ||_{\mathcal{M}(\varOmega)}, && \label{eq:HHNopt1}\\
		|| \bar \Phi ||_{\mathcal{C}(\varOmega)} &\leq \alpha, && \label{eq:HHNopt2} \\
		{ \int_{ \varOmega} \left( \frac{1}{a} \bar \mix v + v' \bar \st \right) \, \textcolor{black}{\rm{dx}} }&=0 \qquad &&\forall v \in H^1(\varOmega), \label{eq:HHNopt3}\\
		{\int_{ \varOmega} \left(- \bar \mix' w + d \bar \st w \right) \, \textcolor{black}{\rm{dx}} }&= { \int_{ \varOmega} w \bar \ct \, \textcolor{black}{\rm{dx}}} \qquad &&\forall w \in L^2(\varOmega),\label{eq:HHNopt4}\\
		{ \int_{ \varOmega} \left( \frac{1}{a} \bar \mixad v + v' \bar \ad \right) \, \textcolor{black}{\rm{dx}}} &=0 \qquad &&\forall v \in H^1(\varOmega), \label{eq:HHNopt5}\\
		{ \int_{ \varOmega} \left( - \bar \mixad' w + d \bar \ad w \right) \,\textcolor{black}{\rm{dx}} }&= { \int_{ \varOmega} w \left( \bar \st - \st_{\operatorname{d}} \right)\, \textcolor{black}{\rm{dx}}} \qquad &&\forall w \in L^2(\varOmega), \label{eq:HHNopt6}\\
		-(\bar \ad, \ct - \bar \ct)_{L^2(\varOmega)} &\leq \alpha \left( ||\ct'||_{\mathcal{M}(\varOmega)} - ||\bar \ct'||_{\mathcal{M}(\varOmega)} \right) \qquad &&\forall \ct \in BV(\varOmega). \label{eq:HHNopt7}
	\end{alignat}
\end{theorem}

{We note that here $(\bar \mixad,\bar \ad)=( \textcolor{black}{\mixad}(\bar \st - \st_{\operatorname{d}}), \textcolor{black}{\ad}(\bar \st - \st_{\operatorname{d}}))$}.

%
\noindent The problem inherits a sparsity structure, where the structure delivers information about the support of $\bar \ct '$, not about the support of the optimal control itself. The support of $\bar \ct '$ indicates the location of the jumping points of the optimal control $\bar \ct \in BV(\varOmega)$.
For the convenience of the reader we recall \cite[Corollary 1]{hafemeyer2019}:

\begin{lemma}\label{lem:HHNsparsity}
	If $\bar \ct$ is optimal for \eqref{eq:HHNproblem}, then there hold
	\begin{align}
		\supp ( (\bar \ct ')^+) &\subset \left\{ x \in \varOmega : \bar \Phi (x) = \alpha \right\}, \label{eq:HHNsparse1} \\
		\supp ( (\bar \ct ')^-) &\subset \left\{ x \in \varOmega : \bar \Phi (x) = - \alpha \right\}, \label{eq:HHNsparse2}
	\end{align}
	where $\bar \ct ' = (\bar \ct ')^+ - (\bar \ct ')^- $ is the Jordan decomposition. Moreover, we have
	\begin{equation}\label{eq:HHNsparse3}
		\supp (\bar \ct ') \subset \left\{ x \in \varOmega : |\bar \Phi (x) | = \alpha \right\} \subset \left\{ x \in \varOmega : \bar \ad (x) = 0 \right\}.
	\end{equation}
\end{lemma}

\section{Variational discretization} \label{sec:HHNVD}

Our aim is to introduce a piecewise constant control approximation, which is fully aligned with the discretization of our state equation. We achieve this by employing variational discretization for our optimal control problem \eqref{eq:HHNproblem} combined with the classical lowest order Raviart-Thomas discretization of the mixed form of the state equation \eqref{eq:HHNPDEmixedweak}. Along with this discrete approach come the facts that our discrete counterpart of \eqref{eq:HHNproblem} still remains infinite-dimensional and that the optimality conditions \eqref{eq:HHNopt3}-\eqref{eq:HHNopt6} remain valid with continuous variables replaced by their discrete analogues.
Then the piecewise constant discretization of the adjoint state $\ad$ in combination with the optimality conditions for the variational discrete problem induce the piecewise constant structure of the control {$\ct_{\operatorname{vd}}$} through the fact that under a natural structural assumption {$\ct_{\operatorname{vd}}'$} is a sum of Dirac measures. This then immediately delivers that the variational discrete control {$\ct_{\operatorname{vd}}$} is piecewise constant.

Let $0=x_0 < x_1 < \ldots < x_N = 1$ be a partition of $\bar \varOmega=[0,1]$. Then for $i=1,\ldots,N$ we define the subintervals $I_i := (x_{i-1},x_i)$ of size $h_i:=x_i - x_{i-1}$ and set $h:= \max_{1 \leq i \leq N} h_i$. Let $\chi_i$ for $i=1,\ldots,N$ be the indicator function of interval $I_i$, i.e.
\begin{equation*}
	\chi_i(x) = \begin{cases}
		1, x\in I_i, \\
		0, \textrm{else}.
	\end{cases}
\end{equation*}
Let $e_j$ for $j=0,\ldots,N$ denote the hat functions, i.e. those functions which are piecewise linear and continuous on the partition, satisfying $e_j(x_i) = \delta_{ij}$ for $i,j = 0, \ldots, N$. 

We introduce the discrete spaces
\begin{equation*}
	P_0 := \Span \left\{ \chi_i : 1 \leq i \leq N \right\}, \qquad 
	P_1 := \Span \left\{ e_j : 0 \leq j \leq N \right\}.
\end{equation*}

\noindent Using these spaces we get the discrete formulation of \eqref{eq:HHNPDEmixedweak}: Find $\st_h = \sum_{i=1}^N \st_i \chi_i \in P_0$, and $\mix_h=\sum_{j=0}^N v_j e_j \in P_1$, such that
{
\begin{subequations}\label{eq:HHNPDEmixedweakdiscrete}
	\begin{alignat}{2}
		\int_{\varOmega} \left( \frac{1}{a} \mix_h v_h + v_h' \st_h \right) \, \textcolor{black}{\rm{dx}}   &= 0 \qquad &&\forall v_h \in P_1,\\
		\int_{\varOmega} \left( - \mix_h' w_h + d \st_h w_h \right) \, \textcolor{black}{\rm{dx}}  &= \int_{\varOmega} w_h \ct  \, \textcolor{black}{\rm{dx}} \qquad &&\forall w_h \in P_0 .
	\end{alignat}
\end{subequations} }

\noindent {For given $\ct$ we write $(\mix_h,\st_h) = (\mix_h(\ct),\st_h(\ct) )$} for the unique solution of the discrete state equation.

\noindent In the present case where $\varOmega \subset \mathds{R}$, the space of Raviart-Thomas elements of lowest order (see e.g. \cite{bahriawaticarstensen,RaviartThomas}) coincides with the chosen \textcolor{black}{pair} $(P_1,P_0)$. Furthermore, we stress that the control space remains $BV(\varOmega)$, so the variational discrete control {$\ct_{\operatorname{vd}}$} is not discretized. 
%
%
\noindent The variational discrete counterpart of \eqref{eq:HHNproblem} then reads
{
\begin{equation}
	\min_{\ct_{\operatorname{vd}} \in BV(\varOmega)} J_h(\ct_{\operatorname{vd}}) \coloneqq \tfrac{1}{2} \| \st_h({\ct_{\operatorname{vd}}}) - \st_{\operatorname{d}} \|_{L^2(\varOmega)}^2 + \alpha \, \| \ct_{\operatorname{vd}}' \|_{\mathcal{M}(\varOmega)}.
	\tag{$P_{\operatorname{vd}}$}
	\label{eq:HHNproblemdiscrete}
\end{equation}
}

\noindent As in \cite[Definition 3.9., Lemma 3.10.]{hafemeyer2019} we define the following projection operator $\Upsilon_h$ and collect its properties.

\begin{lemma}
	For $i=1,\ldots,N$ let the operator $\Upsilon_h$ be defined as:
	\begin{equation*}
	\Upsilon_h : BV(\varOmega) \rightarrow P_0, \qquad \Upsilon_h \ct _{| I_i} \coloneqq \tfrac{1}{h_i} \int_{I_i} \ct (s) \, d s.
	\end{equation*}
	For any $\ct \in BV(\varOmega)$ and $w_h \in P_0$ it holds
	\begin{align}
	(\ct,w_h)_{L^2(\varOmega)} &= (\Upsilon_h \ct, w_h)_{L^2(\varOmega)}, \label{eq:HHNUpsilon1}\\
	|| \ct - \Upsilon_h \ct ||_{L^1(\varOmega)} &\leq h || \ct ' ||_{\mathcal{M}(\varOmega)},\label{eq:HHNUpsilon2}\\
	|| (\Upsilon_h \ct)' ||_{\mathcal{M}(\varOmega)} &\leq || \ct ' ||_{\mathcal{M}(\varOmega)},\label{eq:HHNUpsilon3}\\
	|| \ct - \Upsilon_h \ct ||_{L^{\infty}(\varOmega)} &\leq h || \ct ' ||_{L^{\infty}(\varOmega)},\qquad\qquad \text{provided that } \ct \in W^{1,\infty}(\varOmega).\label{eq:HHNUpsilon4}
	\end{align}
\end{lemma}

The proof can be collected from \cite[Proposition 16]{CasasKruseKunisch} and \cite[Lemma 3.10.]{hafemeyer2019}.

%
We next give the discrete counterpart of Theorem~\ref{thm:HHNuniquesolution}.

\begin{theorem}\label{thm:HHNexistencedisc}
	
	Problem \eqref{eq:HHNproblemdiscrete} admits an optimal solution ${\bar \ct_{\operatorname{vd}} } \in BV(\varOmega)$ and associated unique $(\bar \mix_h,\bar\st_h) \in P_1 \times P_0$. For every solution ${\bar \ct_{\operatorname{vd}} }$, $\Upsilon_h {\bar \ct_{\operatorname{vd}} }$ also solves \eqref{eq:HHNproblemdiscrete} and $\Upsilon_h { \bar \ct_{\operatorname{vd}} } \in P_0$ is unique.
	Furthermore, there exist $C, \textcolor{black}{h_0 >0}$, such that for all $h \in (0,h_0]$ we have 
	\begin{equation}
		|| { \bar \ct_{\operatorname{vd}} } ||_{BV(\varOmega)} \leq C \label{eq:HHNcontrolbounddisc}
	\end{equation}
	for any optimal control $ {\bar \ct_{\operatorname{vd}} }$. 
\end{theorem}
{ \begin{proof}
Since the control {$\ct_{\operatorname{vd}}$} remains continuous, \textcolor{black}{i.e. is not discretized}, and $|\st_h(1)| \ge \frac{1}{2} |\st(1)|$ by \textcolor{black}{Lemma~\ref{lem:HHNLemmaXX}} for $h$ small enough, we can use the proof for existence of solutions from Theorem~\ref{thm:HHNuniquesolution} verbatim, with $\st(1)$ replaced by $\st_h(1)$. \\
%
We stress that uniqueness of the control is not given in this setting, since the control is not discretized, so the control-to-state operator is in general not injective. However by definition of $\Upsilon_h$
\textcolor{black}{ we know that $\Upsilon_h {\bar \ct_{\operatorname{vd}} }$ is also a solution, since $J_h(\Upsilon_h \bar \ct_{\operatorname{vd}}) \le J_h (\bar \ct_{\operatorname{vd}})$}.  
 It is easy to see that {the restriction of the mapping $\ct_{\operatorname{vd}} \mapsto y_h$ to $P_0$ is injective}, so that the quadratic term in $J_h$ now delivers strict convexity of $J_h$ on $P_0$. Therefore, uniqueness of solution in the discrete space $P_0$ is evident. Due to uniqueness of the discrete solution, all projections of solutions ${\bar \ct_{\operatorname{vd}} } \in BV(\varOmega)$ must coincide. \\
For the proof of the boundedness in the $BV$-norm we may proceed along the lines of the respective proof in \cite[Theorem 3.5.]{hafemeyer2019}. First, we recall $|| \bar \ct_{\operatorname{vd}}||_{BV(\varOmega)} = || \bar \ct_{\operatorname{vd}}||_{L^1(\varOmega)}  + || \bar \ct'_{\operatorname{vd}}||_{\mathcal{M}(\varOmega)} $. Since, by optimality of $\bar \ct_{\operatorname{vd}}$ it holds that $J_h(\bar \ct_{\operatorname{vd}} ) \le J_h(0)$, we have 
\begin{equation*}
	 || \bar \ct'_{\operatorname{vd}}||_{\mathcal{M}(\varOmega)} \le \frac{J_h(0)}{\alpha} = \frac{ || \st_h(0) - \st_{\operatorname{d}} ||^2_{L^2(\varOmega)} }{2 \alpha} = \frac{ || \st_{\operatorname{d}} ||^2_{L^2(\varOmega)} }{2 \alpha} .
\end{equation*}
Introducing $\hat \ct := \frac{1}{| \varOmega|} \int_{ \varOmega} \bar \ct_{\operatorname{vd}} \, \textcolor{black}{\rm{dx}}$, we from \cite[Remark 3.50]{AmbrosioFuscoPallara} have
\begin{equation*}
	|| \bar \ct_{\operatorname{vd}} - \hat \ct ||_{L^2(\varOmega)} \le C  || \bar \ct'_{\operatorname{vd}}||_{\mathcal{M}(\varOmega)} ,
\end{equation*}
with $C$ independent of $h$. This implies, exploiting $|\varOmega| = 1$
\begin{equation*}
	|| \bar \ct_{\operatorname{vd}} ||_{L^1(\varOmega)} \le || \bar \ct_{\operatorname{vd}} ||_{L^2(\varOmega)} \le C \frac{ || \st_{\operatorname{d}} ||^2_{L^2(\varOmega)} }{2 \alpha} + |\hat \ct|.
\end{equation*}
To show the boundedness of $\bar \ct_{\operatorname{vd}}$ in $BV(\varOmega)$ it remains to show the boundedness of $|\hat \ct|$. Since the mapping $\ct \mapsto \st_h(\ct)$ is linear \textcolor{black}{and $\hat \ct \in \mathds{R}$}, we have $\st_h(\hat \ct) =\hat \ct \st_h(1)$, and thus
\begin{align*}
	| \hat \ct | \, || \st_h(1)||_{L^2(\varOmega)} &= || \st_h(\hat \ct)||_{L^2(\varOmega)} \\
	&\le  || \st_h(\hat \ct - \bar \ct_{\operatorname{vd}})||_{L^2(\varOmega)} 
	+ \underbrace{ || \st_h(\bar \ct_{\operatorname{vd}}) - \st_{\operatorname{d}}||_{L^2(\varOmega)} }_{\le \sqrt{2 J_h(0)} \, = \, || \st_{\operatorname{d}}||_{L^2(\varOmega)}}
	+  || \st_{\operatorname{d}}||_{L^2(\varOmega)} \\
	&\le || \st_h(\hat \ct - \bar \ct_{\operatorname{vd}})- \st(\hat \ct - \bar \ct_{\operatorname{vd}})||_{L^2(\varOmega)} + || \st(\hat \ct - \bar \ct_{\operatorname{vd}})||_{L^2(\varOmega)} + 2 || \st_{\operatorname{d}}||_{L^2(\varOmega)}.
\end{align*}

\textcolor{black}{
We employ \eqref{eq:HHND} to obtain
\begin{align*}
	|| \st_h(\hat \ct - \bar \ct_{\operatorname{vd}})- \st(\hat \ct - \bar \ct_{\operatorname{vd}})||_{L^2(\varOmega)}  &\le C h \left( || \hat \ct - \bar \ct_{\operatorname{vd}} ||_{L^2(\varOmega)} +  || \st(\hat \ct - \bar \ct_{\operatorname{vd}})||_{H^2(\varOmega)}  \right).
\end{align*} 
}
Altogether, we have, using the continuity of $u \mapsto y(u) \in H^1_0(\varOmega) \cap H^2(\varOmega)$:
\begin{align*}
	| \hat \ct | \, || \st_h(1)||_{L^2(\varOmega)} &\le \textcolor{black}{ C h \, || \hat \ct - \bar \ct_{\operatorname{vd}} ||_{L^2(\varOmega)} + C h \, || \st(\hat \ct - \bar \ct_{\operatorname{vd}})||_{H^2(\varOmega)}  + || \st(\hat \ct - \bar \ct_{\operatorname{vd}})||_{L^2(\varOmega)} + 2 || \st_{\operatorname{d}}||_{L^2(\varOmega)} 
	}\\
	&\le C || \hat \ct - \bar \ct_{\operatorname{vd}} ||_{L^2(\varOmega)}  +2 || \st_{\operatorname{d}}||_{L^2(\varOmega)}\\
	&\le  C ||  \bar \ct'_{\operatorname{vd}} ||_{\mathcal{M}(\varOmega)}+2 || \st_{\operatorname{d}}||_{L^2(\varOmega)}\\
	&\le C.
\end{align*}
Since $|\st_h(1)| \ge \frac{1}{2} |\st(1)|$ for $h$ small enough we conclude $| \hat \ct | \le C$, which finally delivers the boundedness $|| \bar \ct_{\operatorname{vd}} ||_{BV(\varOmega)} \le C$.
\end{proof} }

Analogously to Theorem~\ref{thm:HHNoptcond} from the continuous setting we derive the optimality conditions for \eqref{eq:HHNproblemdiscrete}.

\begin{theorem}\label{thm:HHNoptconddisc}
	The control ${ \bar \ct_{\operatorname{vd}} } \in BV(\varOmega)$ with associated $\textcolor{black}{(\bar \mix_h, \bar \st_h) = ( \mix_h({ \bar \ct_{\operatorname{vd}} }), \st_h({ \bar \ct_{\operatorname{vd}} }))} \in P_1 \times P_0$ is optimal for the problem \eqref{eq:HHNproblemdiscrete} if and only if there exists a unique \textcolor{black}{pair} $(\bar \mixad_h,\bar \ad_h,) \in P_1 \times P_0$, such that $({\bar \ct_{\operatorname{vd}} }, \bar \mix_h, \bar \st_h, \bar \mixad_h, \bar \ad_h)$ and the $P_1$-function $\bar \Phi_h(x) \coloneqq \int_0^x \bar \ad_h(s) \, d s $ satisfy $\bar \Phi_h (1) =0$ as well as  
	
	\begin{alignat}{2}
		\int_{ \varOmega} \bar \Phi_h \, d (\Upsilon_h { \bar \ct_{\operatorname{vd}} }) ' &= \alpha || {\bar \ct_{\operatorname{vd}} } ' ||_{\mathcal{M}(\varOmega)}, && \label{eq:HHNoptdisc1}\\
		|| \bar \Phi_h ||_{\mathcal{C}(\varOmega)} &\leq \alpha, && \label{eq:HHNoptdisc2} \\
		{ \int_{ \varOmega} \left( \frac{1}{a} \bar \mix_h v_h + v_h' \bar \st_h \right) \, \textcolor{black}{\rm{dx}}} &=0 \qquad &&\forall v_h \in P_1, \label{eq:HHNoptdisc3}\\
		{ \int_{ \varOmega} \left( -\bar \mix_h' w_h +d \bar\st_h w_h \right) \, \textcolor{black}{\rm{dx}} }&= {  \int_{ \varOmega} w_h \bar \ct_{\operatorname{vd}} \, \textcolor{black}{\rm{dx}}} \qquad &&\forall w_h \in P_0,\label{eq:HHNoptdisc4}\\
		{ \int_{ \varOmega} \left( \frac{1}{a} \bar \mixad_h v_h +v_h' \bar \ad_h \right) \, \textcolor{black}{\rm{dx}} }&=0 \qquad &&\forall v_h \in P_1, \label{eq:HHNoptdisc5}\\
		{ \int_{ \varOmega} \left(- \bar \mixad_h' w_h + d \bar \ad_h w_h \right) \, \textcolor{black}{\rm{dx}} }&= {  \int_{ \varOmega} w_h \left( \bar \st_h - \st_{\operatorname{d}} \right)\, \textcolor{black}{\rm{dx}}} \qquad &&\forall w_h \in P_0, \label{eq:HHNoptdisc6}\\
		-(\bar \ad_h, \ct - {\bar \ct_{\operatorname{vd}} })_{L^2(\varOmega)} &\leq \alpha \left( ||\ct'||_{\mathcal{M}(\varOmega)} - ||{\bar \ct_{\operatorname{vd}} }'||_{\mathcal{M}(\varOmega)} \right) \qquad &&\forall \ct \in BV(\varOmega). \label{eq:HHNoptdisc7}
	\end{alignat}
\end{theorem}

\begin{proof}
	The optimality of ${\bar \ct_{\operatorname{vd}} } \in BV(\varOmega)$ is equivalent to $0 \in \partial J_h ( { \bar \ct_{\operatorname{vd}} })$. By applying the chain rule and the sum rule, we then deduce
	\begin{equation}
		- \textcolor{black}{ \ad_h}\left( \bar \st_h - \st_{\operatorname{d}} \right) \in \partial \left( \alpha || {\bar \ct_{\operatorname{vd}} }'||_{\mathcal{M}(\varOmega)} \right).
		\label{eq:HHNdiscpartial}
 	\end{equation}
 	{We recall $(\bar \mix_h,\bar \st_h) = \textcolor{black}{( \mix_h({\bar \ct_{\operatorname{vd}} }), \st_h({\bar \ct_{\operatorname{vd}} })) }$, which gives \eqref{eq:HHNoptdisc3} and \eqref{eq:HHNoptdisc4}. With the definition of $(\bar \mixad_h, \bar \ad_h) = \textcolor{black}{( \mixad_h(\bar \st_h - \st_{\operatorname{d}}), \ad_h(\bar \st_h - \st_{\operatorname{d}}))}$ we directly have \eqref{eq:HHNoptdisc5} and \eqref{eq:HHNoptdisc6}, and also that \eqref{eq:HHNoptdisc7} is an equivalent reformulation of \eqref{eq:HHNdiscpartial}.} Inserting $\ct = 2 {\bar \ct_{\operatorname{vd}} }$, $\ct = 0$, $\ct = {\bar \ct_{\operatorname{vd}} } + \tilde \ct$, and $\ct = {\bar \ct_{\operatorname{vd}} } - \tilde \ct$ for arbitrary $\tilde \ct \in BV(\varOmega)$ in \eqref{eq:HHNoptdisc7} delivers
 	\begin{align}
 		- (\bar \ad_h, {\bar \ct_{\operatorname{vd}} })_{L^2(\varOmega)} &= \alpha ||{\bar \ct_{\operatorname{vd}} }'||_{\mathcal{M}(\varOmega)},
	 	\label{eq:HHNdiscequal} \\
	 	\left| (\bar \ad_h,  \ct)_{L^2(\varOmega)} \right| &\leq \alpha || \ct'||_{\mathcal{M}(\varOmega)} \qquad \forall \ct \in BV(\varOmega). \label{eq:HHNdiscinequal}
  	\end{align}
 	Using \eqref{eq:HHNdiscinequal} we conclude
 	\begin{equation*}
	 	\bar \Phi_h (1) = \int_0^1 \bar \ad_h(s) \, ds = (\bar \ad_h,1)_{L^2(\varOmega)} = 0,
 	\end{equation*}
 	and
 	\begin{align*}
 		| \bar \Phi_h(x)| = \left| \int_0^x \bar \ad_h(s) \, ds  \right| = \left| \int_{ \varOmega} \bar \ad_h 1_{(0,x)} \, \textcolor{black}{\rm{dx}} \right| = \left| (\bar \ad_h , 1_{(0,x)})_{L^2(\varOmega)} \right| \leq \alpha.
 	\end{align*}
 	So, we can deduce \eqref{eq:HHNoptdisc2}. We recall the given structure of $\bar \ad_h, \Upsilon_h   {\bar \ct_{\operatorname{vd}} } \in P_0$, and write
 	\begin{equation*}
 		\bar \ad_h = \sum_{i=1}^N \bar \ad_i \chi_i \qquad \text{and} \qquad \Upsilon_h {\bar \ct_{\operatorname{vd}} } = \sum_{i=1}^N \bar \ct_i \chi_i.
 	\end{equation*}
 	Also, we collect the following equality for every gridpoint $x_i \in \varOmega$
 	\begin{equation*}
	 	\bar \Phi_h(x_i) = \int_0^{x_i} \bar \ad_h(s) \, ds = \sum_{j=1}^i \bar \ad_j h_j.
 	\end{equation*}
 	Now, we calculate
 	\begin{align*}
 		\int_{ \varOmega} \bar \Phi_h \, d (\Upsilon_h {\bar \ct_{\operatorname{vd}} } )' 
 		&= \sum_{i=1}^{N-1} (\bar \ct_{i+1} - \bar \ct_i) \bar \Phi_h (x_i) \\
 		&= \sum_{i=1}^{N-1} \bar \ct_{i+1}\bar \Phi_h (x_i)  - \bar \ct_i  \bar \Phi_h (x_{i-1}) - \bar \ct_i \bar \ad_i h_i \\
 		&= \bar \ct_N \bar \Phi_h (x_{N-1}) - \bar \ct_1 \bar \Phi_h (0) - \sum_{i=1}^{N-1} \bar \ct_i \bar \ad_i h_i \\
 		&= \bar \Phi_h(1) - \sum_{i=1}^{N} \bar \ct_i \bar \ad_i h_i \\
 		&= - (\bar \ad_h, \Upsilon_h {\bar \ct_{\operatorname{vd}} })_{L^2(\varOmega)} \\
 		&\stackrel{\eqref{eq:HHNUpsilon1}}{=} - (\bar \ad_h,  {\bar \ct_{\operatorname{vd}} } )_{L^2(\varOmega)} \\
 		&\stackrel{\eqref{eq:HHNdiscequal}}{=} \alpha || { \bar \ct_{\operatorname{vd}} ' } ||_{\mathcal{M}(\varOmega)},
 	\end{align*}
 	where we use $\bar \Phi_h(0)=\bar \Phi_h(1) = 0$. This shows \eqref{eq:HHNoptdisc1} and completes the proof.
\end{proof}


\noindent Furthermore, we deduce a similar sparsity structure as in Lemma~\ref{lem:HHNsparsity} for the continuous problem.

\begin{lemma}\label{lem:HHNsparsitydisc}
	If ${\bar \ct_{\operatorname{vd}} }$ is optimal for \eqref{eq:HHNproblemdiscrete}, then there hold
	\begin{align}
		\supp ( ( \Upsilon_h {\bar \ct_{\operatorname{vd}} } )'_+) &\subset \left\{ x \in \varOmega : \bar \Phi_h (x) = \alpha \right\}, \notag \\ 
		\supp ( (\Upsilon_h {\bar \ct_{\operatorname{vd}} } )'_-) &\subset \left\{ x \in \varOmega : \bar \Phi_h (x) = - \alpha \right\}, \notag  
	\end{align}
	where $ (\Upsilon_h {\bar \ct_{\operatorname{vd}} } )' = (\Upsilon_h {\bar \ct_{\operatorname{vd}} } )'_+ - (\Upsilon_h {\bar \ct_{\operatorname{vd}} } )'_- $ is the Jordan decomposition. Moreover, we have
	\begin{equation*}
		\supp ( (\Upsilon_h {\bar \ct_{\operatorname{vd}} } )') \subset \left\{ x \in \varOmega : |\bar \Phi_h (x) | = \alpha \right\}.
	\end{equation*}
\end{lemma}

The sparsity result can be proven as in the continuous case.

\noindent Even though the control is not discretized, we can deduce information about the structure of the control from the optimality conditions and the sparsity structure, especially properties \eqref{eq:HHNoptdisc2} and \eqref{eq:HHNsparsedisc3}. Let us make the following structural assumption:
\begin{assumption}\label{ass:HHNassumption}
$\bar \ad_h(x) \neq 0$ for all $x \in \varOmega$.
\end{assumption}

\begin{remark}
	This assumption can be easily verified in a numerical setting, hence it is highly practical. Also, later on, we will formulate Assumption \ref{ass:HHNassumptioncont} for the continuous adjoint state $\bar \ad$. We then show that Assumption \ref{ass:HHNassumption} always holds when Assumption \ref{ass:HHNassumptioncont} is satisfied for small enough grid sizes. 
\end{remark}

From $\bar \ad_h(x) \neq 0$ for all $x \in \varOmega$ we deduce $p_i \neq 0 $ for all $i=1,\ldots,N$. Since $\bar\Phi_h(x_i) = \bar\Phi_h(x_{i-1}) + \bar \ad_i h_i$ holds, we get $\bar \Phi_h(x_i) \neq \bar \Phi_h(x_{i-1})$ for $i=1,\ldots,N$. Combining this property with the facts that $\bar\Phi_h \in P_1$ and $|| \bar \Phi_h||_{\mathcal{C}(\varOmega)} \leq \alpha$ we deduce 
\begin{equation*}
	\left\{ x \in \varOmega : | \bar \Phi_h (x) | = \alpha \right\} \subset \left\{ x_i\right\}_{i=1}^N.
\end{equation*}

Furthermore, we have $\bar \Phi_h(x_{i-1}) = \bar \Phi_h(x_{i}) - \bar \ad_i h_i $ and $\bar \Phi_h(x_{i+1}) = \bar \Phi_h(x_{i}) + \bar \ad_{i+1} h_{i+1} $ for any $i=1,\ldots,N$. Now fix some $i\in \left\{ 1,\ldots,N\right\}$ and let $\bar \Phi_h(x_i) = \alpha$. We then use that $p_i \neq 0, p_{i+1}\neq 0$ and $|| \bar \Phi_h||_{\mathcal{C}(\varOmega)} \leq \alpha$ to deduce $\bar \ad_i > 0$ and $\bar \ad_{i+1} < 0$. Similarly, for $\bar \Phi_h(x_i) = -\alpha$ we see $\bar \ad_i < 0$ and $\bar \ad_{i+1} > 0$. Altogether, under Assumption \ref{ass:HHNassumption}, we get 

\begin{equation}\label{eq:HHNsparsedisc3}
\supp ( (\Upsilon_h {\bar \ct_{\operatorname{vd}} } )') \subset \left\{ x \in \varOmega : |\bar \Phi_h (x) | = \alpha \right\} \subset \left\{ x \in \varOmega : \sign(\bar \ad_h(x_-)) \neq \sign(\bar \ad_h(x_+)) \right\},
\end{equation}

{where $v(x_-) := \lim_{s \nearrow x} v(s)$ and $v(x_+) := \lim_{s \searrow x} v(s)$.}

Following \cite{hafemeyer2019}, we can express the optimal discrete control $\Upsilon_h {\bar \ct_{\operatorname{vd}} }$ and its derivative as
\begin{equation*}
	\Upsilon_h {\bar \ct_{\operatorname{vd}} } = \bar a_h + \sum_{i=1}^N \bar c_h^i 1_{(x_i,1)}, \qquad 
	(\Upsilon_h {\bar \ct_{\operatorname{vd}} })' = \sum_{i=1}^N \bar c_h^i \delta_{x_i},
\end{equation*} 
where $\bar a_h \in \mathds{R}, \bar c_h = (\bar c_h^1, \ldots, \bar c_h^N)^{\top} \in \mathds{R}^N$. 
%
%
%
We can determine the coefficients $\bar a_h$ and $\bar c_h$ by solving the finite-dimensional, convex optimization problem
\begin{equation}
	\min_{a_h \in \mathds{R},  c_h \in \mathds{R}^{N}} J_h(a_h,c_h) := \tfrac{1}{2} \| \st_h - \st_{\operatorname{d}} \|_{L^2(\varOmega)}^2 + \alpha \sum_{i=1}^{N} |c_h^i| .
	\tag{$P_h$}
	\label{eq:HHNfindimproblem}
\end{equation}

\subsection{Error estimates} \label{subsec:HHNErrorEstimates}

{
For our numerical analysis we rely on finite element error estimates for mixed approximation of our elliptic two point boundary value problem.
}
Error estimates for mixed finite elements applied to elliptic partial differential equations have been proven e.g. in \cite{BrezziFortin1991,GastaldiNochetto,GiraultRaviart,gong2011mixed,RaviartThomas}.
{
	For the convenience of the reader we in the following cite the respective results, adapted to our 1D situation, which we need for our numerical analysis.\\
	Before we start, let us recall that $BV(\varOmega) \hookrightarrow L^{\infty}(\varOmega)$ continuously for $\varOmega =(0,1)$, so that for given $\ct \in BV(\varOmega)$, the unique solution $(\mix,\st)$ of \eqref{eq:HHNPDEmixedweak} admits the regularity $\mix \in W^{1,\infty}(\varOmega)$ and $\st \in W^{2,\infty}(\varOmega)$. For the solution $(\mixad,\ad)$ of the adjoint equation \eqref{eq:HHNadjointequation} with right hand side $\st - \st_{\operatorname{d}}$ we thus may expect solutions $\mixad \in H^1(\varOmega), \ad \in H^2(\varOmega)$ for $\st_{\operatorname{d}} \in L^2(\varOmega)$, but also higher regularity up to $\mixad \in W^{3,\infty}(\varOmega)$, $\ad \in W^{4,\infty}(\varOmega)$, if e.g. $\st_{\operatorname{d}}\in W^{2,\infty}(\varOmega)$ and $a,d$ are smooth enough, e.g. $a \in W^{3,\infty}, d \in W^{2,\infty}$. The regularity of $\st_{\operatorname{d}}$ thus will have an influence on the quality of uniform error estimates for $\mixad$ and $\ad$, as shall be seen below.  }
	
	As in \cite{DouglasRoberts1985} we define the standard $L^2(\varOmega)$-orthogonal projection $P_h : L^2(\varOmega) \rightarrow P_0$, which for $w \in L^2(\varOmega)$ is defined by
	\begin{equation*}
		(w-P_h w, w_h)\textcolor{black}{_{L^2(\varOmega)}} = 0 \qquad \forall w_h \in P_0.
	\end{equation*}
	Furthermore, we shall use the Fortin projection (see \cite{BrezziFortin1991,DouglasRoberts1985}), defined as $\Pi_h:H^1(\varOmega) \rightarrow P_1$, which for $v\in H^1(\varOmega))$ is defined by 
	\begin{equation*}
		((v- \Pi_h v)', w_h)\textcolor{black}{_{L^2(\varOmega)}}  = 0 \qquad \forall w_h \in P_0.
	\end{equation*}
	%
	For later use we collect the following approximation properties, e.g. from \cite[(A3)]{GastaldiNochetto} and \cite[Section 3]{gong2011mixed} with $2 \le p \le \infty$:
	\begin{alignat}{2}
		||w - P_h w ||_{L^p(\varOmega)} &\leq C h || w' ||_{L^p(\varOmega)} \qquad &&\text{for } w \in W^{1,p}(\varOmega),\notag \\
		||v - \Pi_h v ||_{L^p(\varOmega)} &\leq C h || v' ||_{L^p(\varOmega)} \qquad &&\text{for } v \in W^{1,p}(\varOmega), \notag \\
		||(v - \Pi_h v)' ||_{L^2(\varOmega)} &\leq C h || v'' ||_{L^2(\varOmega)} \qquad &&\text{for }  v' \in H^{1}(\varOmega). \notag
	\end{alignat}
	{
		We repeat some useful a priori error estimates for the mixed finite element approximation. 
\begin{lemma}\label{lem:HHNLemmaXX}		
	Let $\ct \in L^\infty(\varOmega)$ and let $(\mix,\st) = (\mix(\ct), \st(\ct))$ denote the unique solution to \eqref{eq:HHNPDEmixedweak}. Let $(\mix_h(\ct),\st_h(\ct))$ denote the unique corresponding mixed finite element approximation to $(\mix,\st)$. Then, we have 
			\begin{equation} \label{eq:HHND}
				|| \mix - \mix_h(\ct)||_{L^r(\varOmega)} + || \st - \st_h(\ct)||_{L^r(\varOmega)} \le C h \left( || \ct ||_{L^r(\varOmega)} + || \st ||_{W^{2,r}(\varOmega)} \right) \qquad \forall \, 2 \le r < \infty, 
			\end{equation} 
				with $C > 0$ only depending on $r$ and on $\varOmega$, and
			\begin{equation} \label{eq:HHNGN54}
				|| \mix - \mix_h(\ct)||_{L^\infty(\varOmega)} + || \st - \st_h(\ct)||_{L^\infty(\varOmega)} \le C h |\log h| \, || \ct ||_{L^\infty(\varOmega)}.
			\end{equation}
	With $P_h$ and $\Pi_h$ as introduced above, it holds
	\begin{align}
		 || \mix - \mix_h(\ct)||_{L^\infty(\varOmega)} &\le C \left( || \mix - \Pi_h \mix ||_{L^\infty(\varOmega)} + h |\log h| \, ||\ct - P_h \ct ||_{L^\infty(\varOmega)} \right), \label{eq:HHNGN48} \\
		 || P_h \st - \st_h(\ct) ||_{L^\infty(\varOmega)}  &\le C h |\log h| \left( || \mix - \mix_h(\ct)||_{L^\infty(\varOmega)} + || \ct - P_h \ct ||_{L^\infty(\varOmega)} \right) .\label{eq:HHNGN413}
	\end{align}	
Let $\st \in W^{3,\infty}(\varOmega)$, i.e. $\ct \in W^{1,\infty}(\varOmega)$, then 
			\begin{align}
				h \, || \st - \st_h(\ct)||_{L^\infty(\varOmega)} + |\log h|^{-1} || \mix - \mix_h(\ct)||_{L^\infty(\varOmega)} &\le C h^2 || \ct||_{W^{1,\infty}(\varOmega)} . \label{eq:HHNGN52}
			\end{align}
	\end{lemma}
For the proof of \eqref{eq:HHND} we refer to the case $k=0$ in \cite[Theorem 4.2.]{duran1988error}. The results \eqref{eq:HHNGN54}, \eqref{eq:HHNGN48}, \eqref{eq:HHNGN413}, and \eqref{eq:HHNGN52} are \cite[Corollary 5.5., Lemma 4.2., Lemma 4.4., and Corollary 5.2.]{GastaldiNochetto}, respectively.
}

We are now prepared to prove the following a priori error estimate for the state.


\begin{theorem}\label{thm:HHNst}
	{  
		Let $\bar \ct  \in BV(\varOmega)$ denote the unique solution to \eqref{eq:HHNproblem} with associated unique $(\bar \mix, \bar \st)$ from \eqref{eq:HHNPDEmixedweak}, and let $\bar u_{\operatorname{vd}} \in BV(\varOmega)$ denote a solution of \eqref{eq:HHNproblemdiscrete} with uniquely determined discrete optimal $( \bar \mix_h, \bar \st_h)$ from \eqref{eq:HHNPDEmixedweakdiscrete}. Suppose $\st_{\operatorname{d}} \in W^{1,\infty}(\varOmega)$. Then we have 
		\begin{equation}
			|| \bar \st - \bar \st_h||_{L^2(\varOmega)}^2 \leq C \left( h^2 + h \, ||\bar \ct - \bar \ct_{\operatorname{vd}}||_{L^1(\varOmega)} \right) ,
		\end{equation}
	with a constant $C >0$ independent of $h$. 
	}
\end{theorem}

\begin{proof}
	{
	From \eqref{eq:HHNopt7} with $\ct = \bar \ct_{\operatorname{vd}}$ and \eqref{eq:HHNoptdisc7} with $\ct = \bar \ct$ we get 
	\begin{align*}
		-(\bar \ad, \bar\ct_{\operatorname{vd}} - \bar \ct)_{L^2(\varOmega)} &\leq \alpha \left( ||\bar\ct_{\operatorname{vd}}'||_{\mathcal{M}(\varOmega)} - ||\bar \ct'||_{\mathcal{M}(\varOmega)} \right),\\
		-(\bar \ad_h, \bar \ct - \bar \ct_{\operatorname{vd}} )_{L^2(\varOmega)} &\leq \alpha \left( ||\bar \ct'||_{\mathcal{M}(\varOmega)} - ||\bar \ct_{\operatorname{vd}} '||_{\mathcal{M}(\varOmega)} \right),
	\end{align*}
	where we recall that $\bar \ad = \ad(\bar \st - \st_{\operatorname{d}})$, $\bar \ad_h = \ad(\bar \st_h - \st_{\operatorname{d}})$.
	Adding these inequalities delivers
	\begin{align*}
		0 &\le (\bar \ad - \bar \ad_h, \bar \ct_{\operatorname{vd}} - \bar \ct)_{L^2(\varOmega)} \\
		&= (\bar \ad - \ad_h(\bar \st - \st_{\operatorname{d}}),\bar \ct_{\operatorname{vd}} - \bar \ct) + (\ad_h(\bar \st - \st_{\operatorname{d}}) - \bar \ad_h,\bar \ct_{\operatorname{vd}} - \bar \ct) =: \textcolor{black}{(I) + (II)}.
	\end{align*}
	Since $\bar \st - \st_{\operatorname{d}} \in W^{1,\infty}(\varOmega)$ we with \eqref{eq:HHNGN52} obtain 
    \begin{equation*}
    	\textcolor{black}{(I)}  \leq || \bar \ad - \ad_h(\bar \st - \st_{\operatorname{d}}) ||_{L^\infty(\varOmega)} || \bar \ct_{\operatorname{vd}} - \bar \ct ||_{L^1(\varOmega)} \leq Ch \, || \bar \ct_{\operatorname{vd}} - \bar \ct ||_{L^1(\varOmega)} .
    \end{equation*}
Using \eqref{eq:HHNoptdisc3}-\eqref{eq:HHNoptdisc6}, we get
\begin{align*}
	\textcolor{black}{(II)} &= \int_{ \varOmega} -( \bar \mix_h ' -\mix_h(\bar\ct)') (\ad_h(\bar \st - \st_{\operatorname{d}}) -\bar \ad_h) + d (\bar \st_h-\st_h(\bar \ct) ) (\ad_h(\bar \st - \st_{\operatorname{d}}) -\bar \ad_h) \\
	&= \int_{ \varOmega} \frac{1}{a} (\mixad_h(\bar \st - \st_{\operatorname{d}}) -\bar \mixad_h)( \bar \mix_h - \mix_h(\bar\ct)  )  + d (\bar \st_h-\st_h(\bar \ct) ) (\ad_h(\bar \st - \st_{\operatorname{d}}) -\bar \ad_h) \\
	&= \int_{ \varOmega} - (\mixad_h(\bar \st - \st_{\operatorname{d}})' -\bar \mixad_h')(\bar \st_h - \st_h(\bar \ct) )  + d (\bar \st_h-\st_h(\bar \ct) ) (\ad_h(\bar \st - \st_{\operatorname{d}}) -\bar \ad_h)  \\
	&= \int_{ \varOmega} (\bar \st_h - \st_h(\bar \ct) ) ( \bar\st -\st_{\operatorname{d}}) - ( \bar \st_h - \st_h(\bar \ct)) (\bar \st_h -\st_{\operatorname{d}}) \\
	&=  \int_{ \varOmega}   (\bar \st_h - \bar \st + \bar \st - \st_h(\bar\ct)) (\bar\st -\bar \st_h) \\
	&= - ||\bar \st - \bar \st_h||_{L^2(\varOmega)}^2 + \int_{ \varOmega}   ( \bar \st - \st_h(\bar\ct)) (\bar\st -\bar \st_h) \\
	&\le - \frac{1}{2} ||\bar \st - \bar \st_h||_{L^2(\varOmega)}^2 + \frac{1}{2} ||\bar \st - \st_h (\bar \ct) ||_{L^2(\varOmega)}^2.
\end{align*}
Combining the estimates for \textcolor{black}{$(I)$ and $(II)$} we with \eqref{eq:HHND} and \eqref{eq:HHNGN52} obtain 
\begin{align*}
 ||\bar \st - \bar \st_h||_{L^2(\varOmega)}^2 &\le C \left( h^2 + h \, || \bar \ct_{\operatorname{vd}} - \bar \ct ||_{L^1(\varOmega)} \right). 
\end{align*} }
\end{proof}

{
	Let us note that with requiring only $\st_{\operatorname{d}} \in L^\infty(\varOmega)$ we with \eqref{eq:HHNGN54} would have obtained $\textcolor{black}{(I)} \le C h |\log h| \, || \bar \ct_{\operatorname{vd}} - \bar \ct||_{L^1(\varOmega)}$. \\
	Since the variational discrete controls $\bar \ct_{\operatorname{vd}}$ are bounded in $BV(\varOmega)$ w.r.t. $h$, we have
\begin{corollary}\label{cor:HHN}
	With the suppositions of Theorem \ref{thm:HHNst} there holds
	\begin{equation*} \label{eq:stateL2error}
		|| \bar \st - \bar \st_h||_{L^2(\varOmega)} \leq C h^{\frac{1}{2}} .
	\end{equation*}
\end{corollary}
}

We move on to establish an error estimate for the adjoint state.

\begin{theorem}\label{thm:HHNerroradjoint}
	{
		Let the suppositions of Theorem \ref{thm:HHNst}  be satisfied. Let $(\bar \mixad, \bar \ad) = \textcolor{black}{ (\mixad(\bar \st - \st_{\operatorname{d}}),  \ad(\bar \st - \st_{\operatorname{d}}))} $ and $(\bar \mixad_h, \bar \ad_h) = \textcolor{black}{( \mixad_h(\bar \st_h - \st_{\operatorname{d}}),  \ad_h(\bar \st_h - \st_{\operatorname{d}}))}  $. Then, we have
	\begin{align*}
		||\bar \ad - \bar \ad_h ||_{L^{\infty}(\varOmega)}  &\leq C (h + || \bar \st - \bar \st_h||_{L^2(\varOmega)}), \\ 
		||\bar \mixad - \bar \mixad_h ||_{L^{\infty}(\varOmega)} &\leq C (h + || \bar \st - \bar \st_h||_{L^2(\varOmega)}) .
	\end{align*}
}
\end{theorem}

\begin{proof}
	{
		We have 
		\begin{align*}
			||\bar \ad - \bar \ad_h ||_{L^{\infty}(\varOmega)}  &= ||\textcolor{black}{ \ad(\bar \st - \st_{\operatorname{d}}) -\ad_h(\bar \st_h - \st_{\operatorname{d}}) } ||_{L^{\infty}(\varOmega)} \\
			&\le \underbrace{ ||\textcolor{black} {\ad(\bar \st - \st_{\operatorname{d}})} - \ad_h(\bar \st - \st_{\operatorname{d}}) ||_{L^{\infty}(\varOmega)} }_{=:\textcolor{black}{(I)}} + \underbrace{ || \ad_h(\bar \st - \st_{\operatorname{d}}) - \textcolor{black}{ \ad_h(\bar \st_h - \st_{\operatorname{d}}) } ||_{L^{\infty}(\varOmega)} }_{=:\textcolor{black}{(II)}}.
		\end{align*}
	With \eqref{eq:HHNGN52} we see 
	\begin{equation*}
		\textcolor{black}{(I)} \le C h || \bar \st - \st_{\operatorname{d}} ||_{W^{1,\infty}(\varOmega)} \le C h.
	\end{equation*}
	The second term can further be estimated by
	 		\begin{align*}
	 \textcolor{black}{(II)} &\le \underbrace{ || \ad_h(\bar \st - \st_{\operatorname{d}}) - P_h \textcolor{black}{ \ad(\bar \st - \st_{\operatorname{d}}) } ||_{L^{\infty}(\varOmega)} }_{=:\textcolor{black}{(II a)}} + \underbrace{ || P_h \ad(\bar \st - \st_{\operatorname{d}}) - P_h \textcolor{black}{ \ad(\bar \st_h - \st_{\operatorname{d}}) } ||_{L^{\infty}(\varOmega)} }_{=:\textcolor{black}{(II b)}} \\
	 &\qquad + \underbrace{ || P_h \ad(\bar \st_h - \st_{\operatorname{d}}) - \textcolor{black} {\ad_h(\bar \st_h - \st_{\operatorname{d}})} ||_{L^{\infty}(\varOmega)} }_{=:\textcolor{black}{(II c)}}.
	 \end{align*}
 	With \eqref{eq:HHNGN413}, the properties of the $L^2$-projection $P_h$ and \eqref{eq:HHNGN54} we deduce
 	\begin{align*}
 		\textcolor{black}{(II a)} &\le C h |\log h|  \left[ || \textcolor{black}{ \mixad (\bar \st -\st_{\operatorname{d}})} - \mixad_h (\bar \st - \st_{\operatorname{d}}) ||_{L^{\infty}(\varOmega)} + || \bar \st - \st_{\operatorname{d}} -P_h(\bar \st - \st_{\operatorname{d}}) ||_{L^{\infty}(\varOmega)  }  \right] \\
 		&\le C h |\log h|  \left[ |\log h| \, h + h \right] || \bar \st - \st_{\operatorname{d}} ||_{W^{1,\infty}(\varOmega)}.
 	\end{align*}
 	Employing the stability of $P_h$ we get 
 	\begin{align*}
 		\textcolor{black}{(II b)} &\le || \ad (\bar \st - \st_{\operatorname{d}}) - \ad (\bar \st_h - \st_{\operatorname{d}}) ||_{L^{\infty}(\varOmega)} \le C || \ad (\bar \st - \st_{\operatorname{d}}) - \ad (\bar \st_h - \st_{\operatorname{d}}) ||_{H^{1}(\varOmega)} \le C || \bar \st - \bar \st_h ||_{L^2(\varOmega)}.
 	\end{align*}
 	We know that $\bar \st_h = P_h \bar \st_h$ and $|| \st_{\operatorname{d}} - P_h \st_{\operatorname{d}} ||_{L^{\infty}(\varOmega)} \le C h || \st_{\operatorname{d}} ||_{W^{1,\infty}(\varOmega)} $. Combining this with \eqref{eq:HHNGN413} delivers
 	 	\begin{align*}
 		\textcolor{black}{(IIc)} &\le C h |\log h|  \left[ || \textcolor{black}{ \mixad (\bar \st_h -\st_{\operatorname{d}}) } - \mixad_h (\bar \st_h - \st_{\operatorname{d}}) ||_{L^{\infty}(\varOmega)} + || \bar \st_h - \st_{\operatorname{d}} -P_h(\bar \st_h - \st_{\operatorname{d}}) ||_{L^{\infty}(\varOmega)  }  \right] \\
 		&\le C h |\log h|  \left[ h |\log h|  \, || \bar \st_h - \st_{\operatorname{d}} ||_{L^{\infty}(\varOmega)} + h \, || \st_{\operatorname{d}} ||_{W^{1,\infty}(\varOmega)} \right] ,
 	\end{align*}
 where we also have used \eqref{eq:HHNGN54} in the final estimate.
 	Altogether, we for $h |\log h|^2 \le 1$ see 
 	\begin{equation*}
 			||\bar \ad - \bar \ad_h ||_{L^{\infty}(\varOmega)}  \leq C (h + || \bar \st - \bar \st_h||_{L^2(\varOmega)}).
 	\end{equation*}
 	Next, we estimate
 	\begin{align*}
 		||\bar \mixad - \bar \mixad_h ||_{L^{\infty}(\varOmega)}  &= ||\textcolor{black}{ \mixad(\bar \st - \st_{\operatorname{d}}) -  \mixad_h(\bar \st_h - \st_{\operatorname{d}})} ||_{L^{\infty}(\varOmega)} \\
 		&\le \underbrace{ ||\textcolor{black}{ \mixad(\bar \st - \st_{\operatorname{d}}) }- \mixad(\bar \st_h - \st_{\operatorname{d}}) ||_{L^{\infty}(\varOmega)} }_{=: \textcolor{black}{(III)}} + \underbrace{ || \mixad(\bar \st_h - \st_{\operatorname{d}}) - \textcolor{black}{ \mixad_h(\bar \st_h - \st_{\operatorname{d}})} ||_{L^{\infty}(\varOmega)} }_{=:\textcolor{black}{(IV)}}.
 	\end{align*}
 	Since $\bar \mixad = \bar \ad'$ and $\mixad = \ad'$ we have
 	\begin{align*}
 	 \textcolor{black}{(III)} &\le ||\textcolor{black}{ \mixad(\bar \st - \st_{\operatorname{d}})} - \mixad(\bar \st_h - \st_{\operatorname{d}}) ||_{H^{1}(\varOmega)} \le C ||\textcolor{black}{\ad(\bar \st - \st_{\operatorname{d}}) } - \ad(\bar \st_h - \st_{\operatorname{d}}) ||_{H^{2}(\varOmega)} \le C || \bar \st - \bar \st_h ||_{L^2(\varOmega)}.
 	\end{align*}
 	With \eqref{eq:HHNGN48} we get
 	\begin{align*}
 		\textcolor{black}{(IV)} &\le C \left[ \underbrace{ || \mixad (\bar \st_h - \st_{\operatorname{d}}) - \Pi_h \mixad (\bar \st_h - \st_{\operatorname{d}}) ||_{L^{\infty}(\varOmega)} }_{=:\textcolor{black}{(IVa)}} + h | \log h| \, \underbrace{ || \bar \st_h - \st_{\operatorname{d}} - P_h(\bar \st_h - \st_{\operatorname{d}}) ||_{L^{\infty}(\varOmega)} }_{=:\textcolor{black}{(IVb)}}   \right]
 	\end{align*}
 	Due to $| \varOmega| =1$ it holds
 	\begin{align*}
 		\textcolor{black}{(IV a)} \le C h || \mixad (\bar \st_h - \st_{\operatorname{d}}) ||_{W^{1,\infty}(\varOmega)} \le C h || \ad (\bar \st_h - \st_{\operatorname{d}}) ||_{W^{2,\infty}(\varOmega)} \le C h || \bar \st_h - \st_{\operatorname{d}} ||_{L^{\infty}(\varOmega)} \le C h.
 	\end{align*}
 	Also, with $\bar \st_h = P_h \bar \st_h$ we deduce 
 	\begin{equation*}
 		\textcolor{black}{(IV b)} \le Ch || \st_{\operatorname{d}} ||_{W^{1,\infty}(\varOmega)}.
 	\end{equation*}
 	Consequently, for $h$ small enough we have
 	\begin{equation*}
 				||\bar \mixad - \bar \mixad_h ||_{L^{\infty}(\varOmega)} \leq C (h + || \bar \st - \bar \st_h||_{L^2(\varOmega)}) .
 	\end{equation*} 
 }
\end{proof} 

{
\begin{corollary}
	In particular, for $h$ small enough we have with $|| \bar \st - \bar \st_h||_{L^2(\varOmega)} \leq C h^{\frac{1}{2}}$ that 
	\begin{equation*} \label{eq:adjointinftyerror}
			||\bar \ad - \bar \ad_h ||_{L^{\infty}(\varOmega)} + ||\bar \mixad - \bar \mixad_h ||_{L^{\infty}(\varOmega)}  \le C h^{\frac{1}{2}}.
	\end{equation*}
\end{corollary}}

\noindent It is now easy to see the following error estimate for $\bar \Phi$. 

\begin{lemma}\label{lem:HHNPhi}
	Let the suppositions of Theorem \ref{thm:HHNerroradjoint} hold. Then, we have
	\begin{equation*}
		|| \bar \Phi - \bar \Phi_h||_{L^{\infty}(\varOmega)} \leq {C(h + || \bar \st - \bar \st_h||_{L^2(\varOmega)}) }.
	\end{equation*}
\end{lemma}

\begin{proof}
	By inserting the definitions $\bar \Phi(x) = \int_0^x \bar \ad(s) \, ds$ and $\bar \Phi_h = \int_0^x \bar \ad_h(s) \, ds$ it follows directly that 
	\begin{equation*}
		|| \bar \Phi - \bar \Phi_h||_{L^{\infty}(\varOmega)} \leq || \bar \ad - \bar \ad_h||_{L^{1}(\varOmega)},
	\end{equation*}
{
	and due to $|\varOmega|=1$, using \eqref{eq:HHND}, we also have 
	\begin{align*}
		|| \bar \ad - \bar \ad_h||_{L^{1}(\varOmega)} &\leq || \bar \ad - \bar \ad_h||_{L^2(\varOmega)}\\
		&\le ||\textcolor{black}{ \ad(\bar \st - \st_{\operatorname{d}})} - \ad(\bar \st_h - \st_{\operatorname{d}}) ||_{L^{2}(\varOmega)} + || \ad(\bar \st_h - \st_{\operatorname{d}}) - \textcolor{black} {\ad_h(\bar \st_h - \st_{\operatorname{d}})} ||_{L^{2}(\varOmega)} \\ 
		&\le C (  || \bar \st - \bar \st_h||_{L^2(\varOmega)} + h \,  || \bar \st_h - \st_{\operatorname{d}} ||_{L^{2}(\varOmega)}  )\\
		&\le C( || \bar \st - \bar \st_h||_{L^2(\varOmega)} + h). 
	\end{align*}
}
\end{proof}

\noindent {Note that in Lemma \ref{lem:HHNPhi} it is sufficient to require $\st_{\operatorname{d}} \in L^2(\varOmega)$.
Finally, we prove an error estimate for the control under the structural assumption \ref{ass:HHNassumptioncont}, which we formulate next.
}

\begin{assumption}\label{ass:HHNassumptioncont}
	{Let $\bar \ad$ denote the optimal adjoint state.}
	The set $\left\{ x \in \varOmega : \bar \ad(x) = 0 \right\}$ is finite and all roots are simple roots, i.e. if $\bar \ad(x) = 0$, then $\bar \ad'(x) \neq 0$. 
\end{assumption}

We have $\bar \Phi ' = \bar \ad$, so that the inclusion 
\begin{equation*}
\left\{ x \in \varOmega : |\bar \Phi (x)| = \alpha \right\} \subset \left\{ x \in \varOmega : \bar \ad(x) = 0 \right\}
\end{equation*}
holds and $| \bar \Phi|$ attains the value $\alpha$ only at finitely many points.
For notation purposes we set $
 \left\{ x \in \varOmega : \bar \ad(x) = 0 \right\}= \left\{ \hat x_1 , \ldots, \hat x_m\right\},$
with $m=0$ indicating that this set is empty.
\noindent From Lemma~\ref{lem:HHNsparsity} we then deduce that the support of $\bar \ct'$ is finite and we can express $\bar \ct$ as 
\begin{equation*}
	\bar \ct = \bar a + \sum_{i=1}^m \bar c^i 1_{(\hat x_i,1)},
\end{equation*}
where $\bar a \in \mathds{R}$ and $\bar c = (\bar c^1,\ldots,\bar c^m)^{\top}\in\mathds{R}^m$.

\noindent To obtain a convergence result, we in our analysis need to estimate the difference in the jump
points of the optimal control and the corresponding coefficients. We begin by analyzing the jump points across zero of the discrete adjoint state $\bar \ad_h$, which will deliver information about the support of the {finite-dimensional representation $\bar \ct_{h} \coloneqq \Upsilon_h \bar \ct_{\operatorname{vd}} $ of the} variational discrete optimal control.
{
\begin{lemma}
	Let Assumption \ref{ass:HHNassumptioncont} hold. Then there exists $h_0>0$, such that for all $h \in (0,h_0]$ Assumption \ref{ass:HHNassumption} is fulfilled.
\end{lemma}
}
\begin{proof}
{Without loss of generality we assume $m \le N$.}
For $i=1,\ldots,m$ we have {$\hat x_i \in (0,1)$}, since $\bar \Phi(x) = 0$ for $x \in \{0,1 \}$ and $|\bar \Phi (\hat x_i)|=\alpha >0$. {Then, there} exists $R>0$, such that $B_R(\hat x_i) \subset \varOmega$ and all $B_R(\hat x_i)$ are pairwise disjoint. 
Outside of $\cup_i B_R(\hat x_i)$ it holds $|\bar \ad| \geq \epsilon $ for some $\epsilon >0$, {since $\bar \ad$ is continuous}. Furthermore, for an arbitrary but fixed $\hat x \in \cup_i B_R(\hat x_i)$ with $ \hat x \in (\hat x_{i1},\hat x_{i2})$, where $i1,i2 \in \{1,\ldots,m\}$, the sign of $\bar \ad$ does not change in $(\hat x_{i1},\hat x_{i2})${, because all roots are simple roots. Also, $\hat x \in \cup_i B_R(\hat x_i)$ holds for exactly one $i \in \{1,\ldots,m\}$, since $B_R(\hat x_j) \cap B_R(\hat x_k) = \emptyset$ for $j \neq k$.}

From Theorem \ref{thm:HHNerroradjoint} we can deduce
\begin{equation*}
\bar \ad_h \rightarrow \bar \ad \qquad \text{and} \qquad \bar \mixad_h \rightarrow \bar \mixad \qquad \text{uniformly in } h.
\end{equation*}
So, {we may choose $h_0 >0$, such that $|\bar \ad_h(x)| \geq \tfrac{\epsilon}{2}$ for all $h \in (0,h_0]$ and for all $x$ in the complement of $\cup_i B_R(\hat x_i)$}. Consequently, $\bar \ad_h =0$ can not hold outside of $\cup_i B_R(\hat x_i)$. 
{In order to prove that Assumption \ref{ass:HHNassumption} holds, it remains to show that $\bar \ad_h(x) = 0$ can also not hold for $x \in \cup_i B_R(\hat x_i)$.}

Now let us assume that $h_0 < \tfrac{R}{2}$, and choose $i \in \left\{ 1,\ldots, m\right\}$ arbitrary but fixed. By Assumption \ref{ass:HHNassumptioncont} we have $\bar \ad (\hat x_i) = 0$ and $\bar\mixad (\hat x_i) = \bar \ad '(\hat x_i) \neq 0$. Since $\hat x_i$ is the only root of $\bar \ad$ in $B_R(\hat x_i)$ we have $|\bar \mixad(x)| \geq \delta $ for all $x\in B_R(\hat x_i)$ and some $\delta >0$. Without loss of generality we assume $\bar \mixad (x) \geq \delta$ for all $x\in B_R(\hat x_i)$. 
So, {after a possible further reduction of $h_0$} we have $\bar \mixad_h(x) \geq \tfrac{\delta}{2}$ for all $x\in B_R(\hat x_i)$ {for all $h \in (0,h_0]$ by uniform convergence of $\bar \mixad_h$ to $\bar \mixad$}. Now we fix $x_i^-, x_i^+ \in B_R(\hat x_i)$ with 
\begin{equation*}
\bar \ad (x_i^-) <0 \qquad \text{and} \qquad \bar \ad(x_i^+) >0.
\end{equation*}
Due to the uniform convergence we also have $\bar \ad_h(x_i^-) <0$ and $\bar \ad_h(x_i^+)>0$ for $h \in (0,h_0]$, after a possible reduction of $h_0$. Thus, a grid point $x_{j(i)} \in B_R(\hat x_i)$ exists {together with points $x_{j(i)}^{\pm} \in (x_{j(i)} \pm h)$,i.e. $x_{j(i)}^- < x_{j(i)} < x_{j(i)}^+$, such that}
\begin{equation*}
\bar \ad_h (x_{j(i)}^-) <0 \qquad \text{and} \qquad \bar \ad_h(x_{j(i)}^+) >0,
\end{equation*}
{holds, i.e. $\bar \ad_h$ at $x_{j(i)}$ jumps from minus to plus.}
Next, we show that this grid point is the unique jumping point across zero of $\bar \ad_h$ in $B_R(\hat x_i)$. 
Assume there exists at least one other jumping point across zero $x_{k(i)} \neq x_{j(i)}$ in $B_R(\hat x_i)$. Then we without loss of generality 
 { may assume that $\bar \ad_h$ at $x_{k(i)}$ jumps from plus to minus (or is identically zero on $(x_{k(i)} \pm h)$). We thus find $x_{k(i)}^-$ and $x_{k(i)}^+$ in $(x_{k(i)} \pm h)$ with
 	$x_{k(i)}^- < x_{k(i)} < x_{k(i)}^+$, and}
\begin{equation*}
\bar \ad_h (x_{k(i)}^-) \geq 0 \qquad \text{and} \qquad \bar \ad_h(x_{k(i)}^+) \leq 0.
\end{equation*}
We may assume this sign constellation, since if there exists exactly one additional jumping point $x_{k(i)}$ across zero in $B_R(\hat x_i)$, this relation has to hold, and if there exist several additional jumping points, at least one of them satisfies this relation. We choose the hat function $b_{x_{k(i)}}$ with $b_{x_{k(i)}} (x_{k(i)}) = 1$ and $b_{x_{k(i)}} (x_{l}) = 0$ for all $l=1,\ldots,N$ with $l \neq k(i)$. Then we have {using \eqref{eq:HHNoptdisc5}
\begin{align*}
\tfrac{\delta}{2 a} h &\leq \int_{ \varOmega} \frac{1}{a} \bar \mixad_h b_{x_{k(i)}} \, \textcolor{black}{\rm{dx}} = -\int_{ \varOmega} \bar \ad_h b'_{x_{k(i)}} \, \textcolor{black}{\rm{dx}} = - \bar \ad_h (x_{k(i)}^-) + \bar \ad_h (x_{k(i)}^+) \leq 0
\end{align*} }
This contradicts $\delta > 0, a>0, h>0$, so we deduce that $x_{j(i)}$ is unique. Furthermore, by the arguments given this also contradicts {that} $\bar \ad_h(x) = 0$ {holds} for all $x$ on a whole interval of the partition, which is contained in $B_R(\hat x_i)$. Altogether, this implies {$\bar \ad_h(x) \neq 0$ for all $x \in \varOmega$}, i.e. Assumption \ref{ass:HHNassumption} holds {for all $h \in (0,h_0]$ with some $h_0>0$. Moreover, if $\bar \ad$ admits $m$ simple roots $\hat x_1, \ldots, \hat x_m$, $\bar \ad_h$ admits $m$ jump points $x_{j(1)}, \ldots, x_{j(m)}$.}  
\end{proof}
Finally, we want to estimate the distance $|\hat x_i - x_{j(i)}|$. {Since $\bar \ad$ is continuously differentiable} there exists $\xi \in B_R(\hat x_i)$, such that
\begin{alignat}{3}
\bar \ad'(\xi) (\hat x_i -x_{j(i)}) &= \qquad \qquad \qquad \qquad \qquad\;\; \underbrace{\bar \ad (\hat x_i)}_{=0} &&- \bar \ad(x_{j(i)}) \notag\\ 
&= \overbrace{\underbrace{\frac{ \bar \ad_h (x_{j(i)}^-)}{ \bar \ad_h (x_{j(i)}^-) - \bar \ad_h (x_{j(i)}^+) }}_{=: \tilde a} \bar \ad_h (x_{j(i)}^+) + \underbrace{\frac{- \bar \ad_h (x_{j(i)}^+)}{ \bar \ad_h (x_{j(i)}^-) - \bar \ad_h (x_{j(i)}^+) }}_{=:\tilde b} \bar \ad_h (x_{j(i)}^-)} &&- \bar \ad(x_{j(i)}) \notag \\
&=\tilde a \left( \bar \ad_h (x_{j(i)}^+) - \bar \ad(x_{j(i)})  \right) + \tilde b \left(\bar \ad_h (x_{j(i)}^-) - \bar \ad(x_{j(i)})  \right). \notag
\end{alignat}
{Here, $x_{j(i)}^\pm \in (x_{j(i)} \pm h)$ as chosen above.}
Employing $\bar \ad' = \bar \mixad$ and $\tilde b=1-\tilde a$, we get {
\begin{align*}
\delta |\hat x_i - x_{j(i)}| &\leq \bar \mixad(\xi) |\hat x_i - x_{j(i)}| \\
& = \tilde a \left| \bar \ad_h (x_{j(i)}^+) - \bar \ad(x_{j(i)})  \right| + (1-\tilde a) \left|\bar \ad_h (x_{j(i)}^-) - \bar \ad(x_{j(i)})  \right| \\
&\le \tilde a \left\{  \left| \bar \ad_h (x_{j(i)}^+) - \bar \ad(x^+_{j(i)})  \right| +  \left| \bar \ad (x_{j(i)}^+) - \bar \ad(x_{j(i)})  \right|   \right\} \\
&\quad + (1-\tilde a) \left\{ \left|\bar \ad_h (x_{j(i)}^-) - \bar \ad(x^-_{j(i)})  \right|  +  \left|\bar \ad (x_{j(i)}^-) - \bar \ad(x_{j(i)})  \right|  \right\} \\
&\le \tilde a \left\{  || \bar \ad - \bar \ad_h||_{L^\infty(\varOmega)} +  || \bar \ad ||_{W^{1,\infty} (\varOmega)} |x_{j(i)}^+ - x_{j(i)} | \right\} \\
&\quad + (1-\tilde a) \left\{  || \bar \ad - \bar \ad_h||_{L^\infty(\varOmega)} + || \bar \ad ||_{W^{1,\infty} (\varOmega)} |x_{j(i)}^- - x_{j(i)} |  \right\} \\
& \leq || \bar \ad - \bar \ad_h||_{L^\infty(\varOmega)} + C h .
\end{align*} }
{ Finally, Theorem \ref{thm:HHNerroradjoint} gives}
\begin{equation}\label{eq:HHNxdiff}
|\hat x_i - x_{j(i)}| \leq  {C (h + ||\bar \st - \bar \st_h||_{L^2(\varOmega)})}.
\end{equation}

Consequently, $\bar \ct_{h} = \Upsilon_h \bar \ct_{\operatorname{vd}}$ can be represented with $\bar a_h \in \mathds{R}$ and $\bar c_h = (\bar c_h^{j(1)}, \ldots, \bar c_h^{j(m)} )^{\top} \in \mathds{R}^m$ as follows:
\begin{equation*}
	\bar \ct_{h} = \bar a_h + \sum_{i=1}^m \bar c_h^{j(i)} 1_{(x_{j(i)},1)}.
\end{equation*}

Since $m \leq N$ and $x_{j(i)}\subset \left\{ x_i\right\}_{i=1}^N$ we can add zeros in the sum to recover the representation  
$\bar \ct_h = \bar a_h + \sum_{i=1}^N \bar c_h^i 1_{(x_i,1)}$. For now it is more convenient to work with the first representation.

Next, we estimate the differences in the jump heights and the constant coefficient.

\begin{lemma}
	Let Assumption~\ref{ass:HHNassumptioncont} hold. Then there exists $h_0>0$, such that for all $h \in (0,h_0]$ the coefficients of the optimal controls $\bar \ct = \bar a + \sum_{i=1}^m \bar c^i 1_{(\hat x_i,1)}$ and $\bar \ct_{h} = \bar a_h + \sum_{i=1}^m \bar c_h^{j(i)} 1_{(x_{j(i)},1)}$ satisfy
	{
	\begin{align}
		\sum_{i=1}^m | \bar c^i - \bar c_h^{j(i)}| &\leq C \left( h + ||\bar \st - \bar \st_h ||_{L^2(\varOmega)} \right), \label{eq:HHNcdiff}\\
		| \bar a - \bar a_h| &\leq C \left( h + ||\bar \st - \bar \st_h ||_{L^2(\varOmega)} \right), \label{eq:HHNadiff}
	\end{align}	
where $C>0$ denotes a constant independent of $h_0$.
	}
\end{lemma}

\begin{proof}
	We know that there exists a $R>0$, such that the balls $B_{R}(\hat x_i)$ are contained in $\varOmega$ and are pairwise disjoint for $i=1,\ldots,m$. For every $i=1,\ldots,m$ we proceed as follows: 
	Consider a function $g \in \mathcal{C}^{\infty}_c(\varOmega)$, such that $g=1$ on $B_{\tfrac{R}{2}}(\hat x_i)$ and $g=0$ on $\bar \varOmega \setminus \cup_{i=1}^m B_{\tfrac{3}{4}R}(\hat x_i)$. For $h$ small enough we from \eqref{eq:HHNxdiff} and Corollary \ref{cor:HHN} also have $x_{j(i)} \in B_{\tfrac{R}{2}}(\hat x_i)$ for every $i=1,\ldots,m$, {where $x_{j(i)}$ denotes the unique jump point of $\bar \ad_h$ in $B_R(\hat x_i)$}. 
	We have 
	\begin{equation*}
		\bar \ct ' = \sum_{i=1}^m \bar c^i \updelta_{\hat x_i} \qquad \text{and} \qquad \bar \ct_{h} ' = \sum_{i=1}^m \bar c^{j(i)}_h \updelta_{x_{j(i)}},
	\end{equation*}
	so that by the construction of $g$, the definition of the distributional derivative, and the definition of the state equation we get for all $h \in (0,h_0]$ obtain
	\begin{align*}
		| \bar c^i - \bar c^{j(i)}_h | &= \left| \langle \bar \ct' - \bar \ct_{h}' , g \rangle_{\mathcal{M}(\varOmega),\mathcal{C}(\varOmega)} \right|\\
		&=  \left| -(\bar \ct - \bar \ct_{h}, g')_{L^2(\varOmega)} \right|\\
		&\leq \left| (\bar \ct - \bar \ct_{h}, P_h (g'))_{L^2(\varOmega)} \right| + \left| (\bar \ct - \bar \ct_{h}, g' - P_h(g') )_{L^2(\varOmega)} \right|.
	\end{align*}
	The second term can be estimated as follows:
	\begin{align*}
		\left| (\bar \ct - \bar \ct_{h}, g' - P_h(g') )_{L^2(\varOmega)} \right| &\leq || \bar \ct - \bar \ct_{h}||_{L^2(\varOmega)} ||g' - P_h(g') ||_{L^2(\varOmega)}\\
		&\leq \left( ||\bar \ct ||_{L^2(\varOmega)} + || \bar \ct_{h}||_{L^2(\varOmega)} \right) C h || g'' ||_{L^2(\varOmega)} \\
		&\leq C h,
	\end{align*}
	where we use the definition and the properties of $P_h$ together with the bounds $||\bar \ct||_{BV(\varOmega)} \leq C$ and $||\bar \ct_{h}||_{BV(\varOmega)} { = || \Upsilon_h \bar \ct_{\operatorname{vd}} ||_{BV(\varOmega)} \le  || {\bar \ct_{\operatorname{vd}} } ||_{BV(\varOmega)}} \leq C$ from Theorem~\ref{thm:HHNexistencedisc}.

	For the first term we use the definition of $a$ and the definition of $P_h$ to obtain
	
	{
	\begin{align*}
		\left| (\bar \ct - \bar \ct_{h}, P_h (g'))_{L^2(\varOmega)} \right| &= | -b(\bar \mix - \mix(\bar \ct_h), P_h(g') ) + c(\bar \st - \st(\bar \ct_h), P_h(g'))|\\
		&\le  | -b(\bar \mix, P_h(g')-g' ) + c(\bar \st , P_h(g') -g')| \\
		&\qquad +  | -b(\bar \mix - \mix(\bar \ct_h), g' ) + c(\bar \st - \st(\bar \ct_h), g')| \\
		&\le \left| (\bar \ct, P_h (g')-g')_{L^2(\varOmega)} \right| + | -b(\bar \mix - \mix(\bar \ct_h), g' )| + |c(\bar \st - \st(\bar \ct_h), g')| \\
		&\le ||\bar \ct||_{L^2(\varOmega)} || P_h(g')-g'||_{L^2(\varOmega)} + | b(g', \bar \mix - \mix(\bar \ct_h)) | \\
		&\qquad + || d||_{L^\infty(\varOmega)} ||\bar \st - \st(\bar \ct_h) ||_{L^2(\varOmega)} ||g'||_{L^2(\varOmega)}\\
		&\le Ch + \left| \int_{ \varOmega} a (\bar \st - \st(\bar \ct_h)) g''' \right| + C ||\bar \st - \st(\bar \ct_h)||_{L^2(\varOmega)} \\
		&\le C (h + ||\bar \st - \st(\bar \ct_h)||_{L^2(\varOmega)}) \\
		&\le C (h + ||\bar \st - \bar \st_h ||_{L^2(\varOmega)} + ||\bar \st_h - \st(\bar \ct_h)||_{L^2(\varOmega)})\\
		&\le C (h + ||\bar \st - \bar \st_h ||_{L^2(\varOmega)}).
	\end{align*}}
	\noindent Summarizing, we see $| \bar c^i - \bar c^{j(i)}_h | \leq {C (h + ||\bar \st - \bar \st_h ||_{L^2(\varOmega)})}$ for every $i=1,\ldots,m$, which delivers \eqref{eq:HHNcdiff}. 
	To see \eqref{eq:HHNadiff}, we can adapt the proof of \cite[Lemma 4.9.]{hafemeyer2019} to our setting.
	{In particular, we have
		\textcolor{black}{ 
	\begin{align*}
		\bar \ad - \bar \ad_h &= \ad ( \st (\bar \ct) - \st_{\operatorname{d}}) -  \ad_h ( \st_h (\bar \ct_h) - \st_{\operatorname{d}} ) \\
		&= \ad \left( \st \left(\bar a + \sum_{i=1}^m \bar c^i 1_{(\hat x_i,1)} \right)  \right) -  \ad_h \left(  \st_h \left(  \bar a_h + \sum_{i=1}^m \bar c_h^{j(i)} 1_{(x_{j(i)},1)}  \right)   \right) - \left(   \ad -  \ad_h  \right)(\st_{\operatorname{d}}) \\
		&= (\bar a - \bar a_h)  \ad ( \st(1)) + \bar a_h ( \ad ( \st) -  \ad_h ( \st_h))(1) + \sum_{i=1}^m (\bar c^i - \bar c_h^{j(i)}) \ad (\st (1_{(\hat x_i,1)})) \\
		&\quad + \sum_{i=1}^m \bar c_h^{j(i)} (\ad ( \st) - \ad_h (\st_h))(1_{(\hat x_i,1)}) + \sum_{i=1}^m \bar c_h^{j(i)} (  \ad_h (\st_h(1_{(\hat x_i,1)} - 1_{(x_{j(i)},1)})))- \left(  \ad - \ad_h  \right)(\st_{\operatorname{d}}).
	\end{align*}
}
By Theorem \ref{thm:HHNerroradjoint} the means of $\bar \ad$ and $\bar \ad_h $ vanish, so by integration we get
\textcolor{black}{
\begin{align*}
0&=	(\bar a - \bar a_h) \int_{ \varOmega} \ad ( \st(1)) \textcolor{black}{\rm{dx}} + \bar a_h \int_{ \varOmega} ( \ad ( \st) -  \ad_h ( \st_h))(1) \textcolor{black}{\rm{dx}} + \sum_{i=1}^m (\bar c^i - \bar c_h^{j(i)}) \int_{ \varOmega}  \ad ( \st (1_{(\hat x_i,1)})) \textcolor{black}{\rm{dx}} \\
&\quad + \sum_{i=1}^m \bar c_h^{j(i)} \int_{ \varOmega} ( \ad ( \st) -  \ad_h ( \st_h))(1_{(\hat x_i,1)}) \textcolor{black}{\rm{dx}} + \sum_{i=1}^m \bar c_h^{j(i)} \int_{ \varOmega}   \ad_h ( \st_h(1_{(\hat x_i,1)} - 1_{(x_{j(i)},1)})) \textcolor{black}{\rm{dx}} \\
&\quad - \int_{ \varOmega} \left(   \ad -  \ad_h  \right)(\st_{\operatorname{d}}) \textcolor{black}{\rm{dx}}.
\end{align*}
}
Let us show a useful equality for general $\ct_1,\ct_2 \in L^2(\varOmega)$:
\textcolor{black}{
\begin{align*}
	\int_{ \varOmega}  \ad ( \st (\ct_1)) \ct_2  \textcolor{black}{\rm{dx}} &\stackrel{\eqref{eq:HHNopt4}}{=} \int_{ \varOmega} -  \mix'(\ct_2)  \ad ( \st(\ct_1)) + d \st(\ct_2)  \ad ( \st (\ct_1)) \textcolor{black}{\rm{dx}} \\
	&\stackrel{\eqref{eq:HHNopt5}}{=} \int_{ \varOmega} \frac{1}{a}  \mixad (\st (\ct_1))  \mix(\ct_2) + d  \st(\ct_2)  \ad ( \st (\ct_1)) \textcolor{black}{\rm{dx}} \\
	&\stackrel{\eqref{eq:HHNopt3}}{=} \int_{ \varOmega} -  \mixad'( \st (\ct_1))  \st(\ct_2) + d  \st(\ct_2)  \ad ( \st(\ct_1)) \textcolor{black}{\rm{dx}} \\
	&\stackrel{\eqref{eq:HHNopt6}}{=} \int_{ \varOmega}  \st (\ct_1) \st (\ct_2) \textcolor{black}{\rm{dx}}.
\end{align*}
}
The same equation holds true for the discrete setting by making use of \eqref{eq:HHNoptdisc3}-\eqref{eq:HHNoptdisc6}.\\
With $\ct_1 = \ct_2 =1$ we see $\int_{ \varOmega} \textcolor{black}{ \ad (\st(1))} \textcolor{black}{\rm{dx}}  = || \textcolor{black}{ \st(1)} ||_{L^2(\varOmega)}^2$ and knowing that $\textcolor{black}{ \st (1)} \neq 0$ we get
\textcolor{black}{
\begin{align} \label{eq:HHNaminusah}
	|\bar a - \bar a_h| &\le  ||  \st (1)||_{L^2(\varOmega)}^{-2} \left( 
	 |\bar a_h| \, || ( \ad ( \st) -  \ad_h ( \st_h))(1) ||_{L^1(\varOmega)} + \sum_{i=1}^m | \bar c^i - \bar c_h^{j(i)} | \, ||  \ad ( \st (1_{(\hat x_i,1)})) ||_{L^1(\varOmega)} \right. \\
	 &\left. \quad\;\; + \sum_{i=1}^m | \bar c_h^{j(i)} | \,  || ( \ad ( \st) -  \ad_h ( \st_h))(1_{(\hat x_i,1)}) ||_{L^1(\varOmega)} + \sum_{i=1}^m | \bar c_h^{j(i)} | \,  \left| \int_{ \varOmega}  \ad_h ( \st_h(1_{(\hat x_i,1)} - 1_{(x_{j(i)},1)})) \textcolor{black}{\rm{dx}} \right|   \right. \notag \\ 
	 &\left.  \phantom{\sum_{i=1}^m} + || \left(   \ad -  \ad_h  \right)(\st_{\operatorname{d}}) ||_{L^1(\varOmega)}.
	 \right). \notag
\end{align}   
}
Using continuity of $\bar \st_h \mapsto \bar \ad_h$, $|\varOmega| =1$, and \eqref{eq:HHND} with $r=2$, we obtain
\textcolor{black}{
\begin{align*}
	 || ( \ad (  \st) -  \ad_h ( \st_h))(1) ||_{L^1(\varOmega)} &\le  || ( \ad (\st) - \ad_h ( \st))(1) ||_{L^1(\varOmega)} +  || ( \ad_h ( \st) -  \ad_h ( \st_h))(1) ||_{L^1(\varOmega)}
	 \\
	 &\le  || ( \ad -  \ad_h) ( \st (1)) ||_{L^2(\varOmega)} +  C ||  \st(1) -  \st_h (1) ||_{L^1(\varOmega)}
	 \\
	 &\le C h \,  ||  \st (1) ||_{L^2(\varOmega)} + C  ||  \st(1) -  \st_h (1) ||_{L^2(\varOmega)}\\
	 &\le C h .
\end{align*}
}
Since $1_{(\hat x_i,1)} \in L^\infty (\varOmega)$, an analogous estimation can be done using \eqref{eq:HHND} with $r=2$ for every $i$ to see that $|| \textcolor{black}{( \ad ( \st) -  \ad_h ( \st_h))(1_{(\hat x_i,1)})} ||_{L^1(\varOmega)} \le C h $.  Furthermore, with the helpful equality from above in the discrete setting, we have for every $i$
\textcolor{black}{
\begin{align*}
	\left| \int_{ \varOmega}  \ad_h ( \st_h(1_{(\hat x_i,1)} - 1_{(x_{j(i)},1)})) \textcolor{black}{\rm{dx}} \right| &= \left| \int_{ \varOmega}   \st_h(1_{(\hat x_i,1)} - 1_{(x_{j(i)},1)})   \st_h(1)  \textcolor{black}{\rm{dx}} \right| 
	\\
	&= \left| \int_{ \varOmega}   \ad_h ( \st_h(1))  (1_{(\hat x_i,1)} - 1_{(x_{j(i)},1)})) \textcolor{black}{\rm{dx}} \right| 
	\\
	&\le ||  \ad_h( \st_h(1)) ||_{L^\infty (\varOmega)} | \hat x_i - x_{j(i)} |
	\\
	&\le 2 ||  \ad( \st(1)) ||_{L^\infty (\varOmega)} | \hat x_i - x_{j(i)} |
	\\
	&\le C | \hat x_i - x_{j(i)} |,
\end{align*}
}
where we used that $h$ is chosen small enough and $\textcolor{black}{ \ad( \st(1))} \neq 0$.
Also, with $|\varOmega| =1$ and \eqref{eq:HHND} with $r=2$ we have
\begin{equation*}
	|| \textcolor{black}{\left(  \ad -  \ad_h  \right)}(\st_{\operatorname{d}}) ||_{L^1(\varOmega)} \le || \textcolor{black}{\left(   \ad -  \ad_h  \right)}(\st_{\operatorname{d}}) ||_{L^2(\varOmega)} \le C h \, ||\st_{\operatorname{d}} ||_{L^2(\varOmega)}  \le C h .
\end{equation*}
Plugging all of this into \eqref{eq:HHNaminusah}, and employing $|\bar c_h|_1 = \sum_{i=1}^m | \bar c_h^{j(i)} | $, delivers
\begin{equation*}
	|\bar a - \bar a_h | \le C h \left( |\bar a_h| +  |\bar c_h|_1 +1 \right) + C \left(  \sum_{i=1}^m | \bar c^i - \bar c_h^{j(i)} | + |\bar c_h|_1 \sum_{i=1}^m |\hat x_i - x_{j(i)} | \right).
\end{equation*}
From the definition of $\bar \ct_h$ we obtain
\begin{equation*}
	\frac{1}{2} || \bar a_h \textcolor{black}{ \st_h(1)} + \sum_{i=1}^m \bar c^{j(i)}_h \textcolor{black}{ \st_h(1_{(x_{j(i)},1)}) } - \st_{\operatorname{d}} ||^2_{L^2(\varOmega)} + \alpha |\bar c_h|_1 = J_h(\bar a_h, \bar c_h) \le J_h(0) = J(0).
 \end{equation*}
This delivers $|\bar c_h|_1 \le \frac{J(0)}{\alpha}$.
Also, for $h$ small enough
\begin{align*}
	\frac{1}{2} | \bar a_h| \, || \st(1) ||_{L^1(\varOmega)} &\le | \bar a_h | \, || \st_h(1)||_{L^1(\varOmega)} =  || \st_h(\bar a_h)||_{L^1(\varOmega)} \\
	&\le  || \st_h (\bar \ct_h) - \st_{\operatorname{d}} ||_{L^1(\varOmega)} + || \st_{\operatorname{d}} - \st_h\left( \sum_{i=1}^m \bar c_h^{j(i)}  1_{(x_{j(i)},1)} \right)||_{L^1(\varOmega)} \\
	&\le C + || \st_{\operatorname{d}} ||_{L^1(\varOmega)}  + || \sum_{i=1}^m \bar c_h^{j(i)} \st_h ( 1_{(x_{j(i)},1)} ) ||_{L^1(\varOmega)}   \\
	&\le C + \sum_{i=1}^m | \bar c_h^{j(i)} | \max_{1 \leq i \leq m} || \st_h( 1_{(x_{j(i)},1)} ) ||_{L^1(\varOmega)} \\
	&\le C (1 + \frac{2 J(0)}{\alpha} \max_{1 \leq i \leq m}  ||  \st (1_{(x_{j(i)},1)} ) ||_{L^1(\varOmega)} ) \\
	&\le C.
\end{align*}
This implies $|\bar a_h| \le C$ uniformly in $h$. 
Using \eqref{eq:HHNxdiff} and \eqref{eq:HHNcdiff}, we finally obtain
\begin{equation*}
	|\bar a - \bar a_h| \le C ( h + ||\bar \st - \bar \st_h ||_{L^2(\varOmega)}).
\end{equation*} }
\end{proof}

\noindent With the previous results we now have everything available to prove the convergence order $\mathcal{O}(h)$ for the optimal control. 
\begin{theorem} \label{thm:HHNcterror}
	Let Assumption~\ref{ass:HHNassumptioncont} hold. Then there exists $h_0>0$, such that for all $h \in (0,h_0]$ we have 
	{
	\begin{align*}
		|| \bar \ct - \bar \ct_{h} ||_{L^1(\varOmega)} &\leq C \left( h+ ||\bar \st - \bar \st_h ||_{L^2(\varOmega)} \right),\\
		|| \bar \ct' - \bar \ct_h' ||_{(W^{1,\infty}(\varOmega))'} &\leq C \left( h + ||\bar \st - \bar \st_h ||_{L^2(\varOmega)} \right).
	\end{align*}
	}
\end{theorem}

\begin{proof}
	We combine $|\varOmega|=1$, \eqref{eq:HHNxdiff}, \eqref{eq:HHNcdiff}, and \eqref{eq:HHNadiff} to get
	\begin{align*}
		|| \bar \ct - \bar \ct_{h} ||_{L^1(\varOmega)} &= \int_{ \varOmega}{ \left|  \bar a - \bar a_h + \sum_{i=1}^m \left( \bar c^i 1_{(\hat x_i,1)} - \bar c_h^{j(i)} 1_{(x_{j(i)},1)} \right) \right|  \, \textcolor{black}{\rm{dx}}} \\
		&\leq | \bar a - \bar a_h| |\varOmega| + \int_{\varOmega}{ \left|  \sum_{i=1}^m \bar c^i \left( 1_{(\hat x_i,1)} - 1_{(x_{j(i)},1)} \right) \right|  \, \textcolor{black}{\rm{dx}}} + \int_{\varOmega}{ \left| \sum_{i=1}^m \left( \bar c^i - \bar c_h^{j(i)} \right) 1_{(x_{j(i)},1)} \right|  \, \textcolor{black}{\rm{dx}} }\\
		&\leq | \bar a - \bar a_h | + \sum_{i=1}^m |\bar c^i| \, \underbrace{|| 1_{(\hat x_i,1)} - 1_{(x_{j(i)},1)}||_{L^1(\varOmega)}}_{= |\hat x_i - x_{j(i)} |} + \sum_{i=1}^m | \bar c^i - \bar c_h^{j(i)} |\, ||  1_{(x_{j(i)},1)}||_{L^1(\varOmega)}\\
		&\leq {C \left( h  + ||\bar \st - \bar \st_h ||_{L^2(\varOmega)} \right).} 
	\end{align*}
	{
	Also, we have
	\begin{align*}
		|| \bar \ct' - \bar \ct_h' ||_{(W^{1,\infty}(\varOmega))'} &=  \sup_{||f||_{W^{1,\infty}(\varOmega)} \leq 1} \int_{ \varOmega} 	f \, d \left( \bar \ct' - \bar \ct_h' \right) \\
		&=  \sup_{||f||_{W^{1,\infty}(\varOmega)} \leq 1} \int_{ \varOmega} 	f \, d \left( \sum_{i=1}^m \bar c^i \updelta_{\hat x_i} - \sum_{i=1}^m \bar c^{j(i)}_h \updelta_{x_{j(i)}} \right)\\
		&=  \sup_{||f||_{W^{1,\infty}(\varOmega)} \leq 1}  \sum_{i=1}^m  \underbrace{\vphantom{\left(\bar c_h^i\right)}\bar c^i}_{\in \, \mathds{R}} \underbrace{\left(f(\hat x_i) - f(x_{j(i)}) \right)}_{\leq \,L_f |\hat x_i - x_{j(i)}|, \text{with } L_f\le 1 } + \left( \bar c^i- \bar c^{j(i)}_h \right) \underbrace{\vphantom{\left(\bar c_h^i\right)}f(x_{j(i)})}_{\leq 1}\\
		&\le  \sup_{||f||_{W^{1,\infty}(\varOmega)} \leq 1} \left\{ \max_{1 \leq i \leq m} |f(\hat x_i) - f(x_{j(i)})| \, || \bar c||_{l^1(\varOmega)} + || \bar c - \bar c_h ||_{l^1(\varOmega)} \max_{1 \leq i \leq m} | f(x_{j(i)}) |  \right\}
		\\
		&\le \sup_{||f||_{W^{1,\infty}(\varOmega)} \leq 1}  L_f \max_{1 \leq i \leq m} |\hat x_i - x_{j(i)}| \, || \bar c||_{l^1(\varOmega)} + || \bar c - \bar c_h ||_{l^1(\varOmega)}
		\\
		&\le C \left( h  + ||\bar \st - \bar \st_h ||_{L^2(\varOmega)} \right),
	\end{align*}
where $L_f \le 1$ denotes the Lipschitz constant of $f$.
	}

\end{proof}

{Combining this result with Theorem \ref{thm:HHNst}, we under the structural Assumption~\ref{ass:HHNassumptioncont} deduce improved error estimates.
\begin{lemma}
	Let Assumption \ref{ass:HHNassumptioncont}hold. Then there exists $h_0>0$, such that for all $h \in (0,h_0]$ we have the following error estimates, where $C>0$ denotes a constant independent of $h$.
	\begin{alignat}{3}
		|| \bar \st - \bar \st_h||_{L^2(\varOmega)} &\le C h, \quad\;\;\; || \bar \ad - \bar \ad_h||_{L^\infty(\varOmega)} &&\le C h , \qquad \quad\;\;\;		
		|| \bar \mixad - \bar \mixad_h||_{L^\infty(\varOmega)} &&\le C h , \notag \\ \qquad || \bar \Phi - \bar \Phi_h||_{L^\infty(\varOmega)} &\le C h , \qquad\quad\;		| \hat x_i - x_{j(i)}| &&\le Ch , 
		\qquad\qquad  \sum_{i=1}^m |\bar c^i - \bar c^{j(i)}_h | &&\le C h , \notag \\
		| \bar a - \bar a_h| &\le C h , 
		\qquad || \bar \ct - \bar \ct_h ||_{L^1(\varOmega)} &&\le C h , 
		\qquad
		|| \bar \ct' - \bar \ct_h' ||_{(W^{1,\infty}(\varOmega))'} &&\le C h . \notag
	\end{alignat}
\end{lemma}
\begin{proof}
	Since $\bar \ct_h$ is a solution of \eqref{eq:HHNproblemdiscrete} and Theorem \ref{thm:HHNst} holds for all solutions $\bar \ct_{\operatorname{vd}}$ of \eqref{eq:HHNproblemdiscrete}, we can plug the result from Theorem \ref{thm:HHNcterror} into the error estimate for the state from Theorem \ref{thm:HHNst} and see with Young's inequality 
	\begin{equation*}
		|| \bar \st - \bar \st_h||^2_{L^2(\varOmega)}  \le C (h^2 + h \, ||\bar \ct - \bar \ct_{\operatorname{vd}}||_{L^1(\varOmega)}   ) \le C (h^2 + h \, || \bar \st - \bar \st_h||_{L^2(\varOmega)} ) \le C h^2 + \frac{1}{2} || \bar \st - \bar \st_h||^2_{L^2(\varOmega)} .
	\end{equation*}
	This delivers 
	\begin{equation*}
		|| \bar \st - \bar \st_h||_{L^2(\varOmega)} \le C h. 
	\end{equation*}
	Now we can use the above inequality in Theorem \ref{thm:HHNerroradjoint}, Lemma \ref{lem:HHNPhi}, \eqref{eq:HHNxdiff}, \eqref{eq:HHNcdiff}, \eqref{eq:HHNadiff} and Theorem \ref{thm:HHNcterror} to derive the remaining error estimates. 
\end{proof}
}

\section{Computational results} \label{sec:HHNCompRes}

We can represent the mixed formulation of the discrete state equation \eqref{eq:HHNPDEmixedweakdiscrete} by the following matrix equation:

\begin{equation} \label{eq:HHNmatrixeq}
	\begin{pmatrix} 
		A & B \\
		B^{\top} & {D}
	\end{pmatrix}
	\begin{pmatrix}
		\textbf{\mix} \\
		\textbf{\st}
	\end{pmatrix}
	= 
	\begin{pmatrix}
		0 \\
		-\textbf{\ct} 
	\end{pmatrix},
\end{equation}
{
where for the given grid $0 = x_0 < x_1 < \ldots < x_N =1$ with associated spaces $P_0 = \Span\{ \chi_i : 1 \le i \le N \}$ and $P_1 = \Span\{ e_j : 0 \le j \le N \}$ the matrix entries are given by 
\begin{align*}
	A &= (a_{i,j})_{i,j=0}^N,               &&a_{i,j} = a(e_i,e_j), \\
	B &= (b_{i,j})_{i=0, j=1}^N,           &&b_{i,j} = b(e_i, \chi_j), \\
	D &= (d_{i,j})_{i,j=1}^N, 				 &&d_{i,j} = c(\chi_i, \chi_j).
\end{align*}
Then $A$ is symmetric positive definite, $D$ is symmetric positive semi definite and $B$ has rank $N$ (given $\ker B = \mathbb{1}$).}
$A \in \mathds{R}^{(N+1) \times (N+1)}$ and $B\in \mathds{R}^{(N+1) \times N}$ have the entries
\begin{equation*}
	A = \begin{pmatrix}
		\tfrac{1}{3} h_1 & \tfrac{1}{6} h_1         & 0                   & \ldots                    & 0 \\
		\tfrac{1}{6} h_1 & \tfrac{1}{3} (h_1 + h_2) & \tfrac{1}{6} h_2     &  \ddots                   & \vdots \\
		0               & \ddots                  & \ddots              & \ddots                    &   0     \\
		\vdots         & \ddots                  & \tfrac{1}{6} h_{N-1} & \tfrac{1}{3} (h_{N-1}+h_N) & \tfrac{1}{6} h_N \\
		0               & \ldots                  &  0                  &  \tfrac{1}{6} h_N          & \tfrac{1}{3} h_N
	\end{pmatrix}
	, \quad
	B = \begin{pmatrix}
		-1     & 0      & \ldots & 0 \\
		1      & - 1    &        & \vdots \\
		0      & \ddots & \ddots & 0     \\
		\vdots &        & \ddots & -1 \\
		0      & \ldots & 0      & 1
	\end{pmatrix}
	.
\end{equation*}
The vectors contain the coefficients $\textbf{\mix} = \begin{pmatrix}
\mix_0, \ldots, \mix_N
\end{pmatrix}^{\top} \in \mathds{R}^{N+1}, 
\textbf{\st} = \begin{pmatrix}
\st_1, \ldots, \st_N
\end{pmatrix}^{\top} \in \mathds{R}^N, $ and the evaluation of the BV-function
$\textbf{\ct} = \begin{pmatrix}
\ct_1, \ldots, \ct_N
\end{pmatrix}^{\top} \in \mathds{R}^N$,  where $\ct_j := \int_{x_{j-1}}^{x_j} \ct $ for $j=1,\ldots,N$. With our knowledge about the structure of $\ct$ we obtain $\ct_j = (a_h + \sum_{i=1}^{j-1} c_h^i) \,h_j$ for $j=1,\ldots,N$, where $h_j = x_{j} - x_{j-1}$.

We use \eqref{eq:HHNmatrixeq} to get $\textbf{\mix} = -A^{-1} B\textbf{\st}$ and {$ (B^{\top} A^{-1}  B + D) \textbf{\st} = \textbf{\ct}$, so that $\textbf{\st} = (B^{\top}A^{-1}B + D)^{-1}  \textbf{\ct}$.} Then we insert this into \eqref{eq:HHNfindimproblem} and obtain: 
\begin{equation}	\label{eq:HHNfindimproblem2} 
	\min_{a_h \in \mathds{R},  c_h \in \mathds{R}^{N-1}} f(a_h,c_h) :=\tfrac{1}{2} \| {(B^{\top}A^{-1}B +D)^{-1}} \textbf{\ct} - \st_{\operatorname{d}} \|_{L^2(\varOmega)}^2 + \alpha \sum_{i=1}^{N-1} |c_h^i|,
	\tag{$\hat P_h$}
\end{equation}
where $\textbf{\ct} = \textbf{\ct}(a_h,c_h) $ {is now considered a function of $a_h$ and $c_h$.}
%

\subsection{Optimization algorithm}
{It follows from our variationally discrete approach
that the support of $\bar \ct_{h}'$ is a subset of the grid points $\left\{ x_i\right\}_{i=1}^N$, so we don't need to approximate the support like e.g. needed to be done in the classical fully discrete approach with piecewise constant controls in \cite{hafemeyer2019}.}
We start the algorithm with an empty support set and then update the set of support points in each outer iteration, where we will determine the grid points, at which the control is actually supported. 

We define $m_{k}$ as the cardinality of support points in iteration $k$ and $t_k$ the sorted vector of all support points in iteration $k$.
The outer iteration should be terminated if the support points satisfy
\begin{equation}
	m_k = m_{k-1} \qquad \textrm{and} \qquad \| t_k - t_{k-1}\|_2 \leq \epsilon.
	\label{eq:HHNtermination}
	\tag{$T_1$}
\end{equation}
Here, the second condition only needs to be checked if the first condition is fulfilled, to ensure that the support points are identical in both iterations. In \cite{hafemeyer2019} cycling of the outer iteration is reported. We also observe this in our numerical experiments. We note that in \cite{trautmannwalter2021} a solution strategy is proposed which seems to avoid cycling similar minimization problems. To detect cycling we insert a second set of termination conditions:
\begin{equation}
	m_k = m_{k-1} = m_{k-2} \qquad \textrm{and} \qquad \| t_k - t_{k-2} \|_2 \leq \epsilon \qquad \textrm{and} \qquad f^k < f^{k-1},
	\label{eq:HHNtermination2}
	\tag{$T_2$}
\end{equation}
where $f^k := f(a_h^k, c_h^k,r^k,s^k)$.
This leads to the following algorithm for the numerical solution of \eqref{eq:HHNfindimproblem2}:\\

\begin{algorithm}[H]
	\DontPrintSemicolon
	\SetKwInOut{Input}{input}
	\SetKwInOut{Output}{output}
	
	\Input{$m_0 \in \mathds{R}, t_0 \in \mathds{R}^{m_0}, \epsilon >0$}
	\For{$k=0,1,\ldots$}    
	{ 
		\If{ \eqref{eq:HHNtermination} or \eqref{eq:HHNtermination2} holds}
		{
			$m:= m_k$, $\bar x_h := t_k$, \; 
			extract $(\bar a_h, \bar c_h)$ from $\ct_h^k$ \; 
			STOP
		}
		
		Obtain $(\ct_h^k,\st_h^k,\ad_h^k) $ by solving \eqref{eq:HHNfindimproblem2} {on $t_k$.} \;
		Compute $t_{k+1} \in \mathds{R}^{m_{k+1}}$ from $\ad_h^k$.
	}
	\Output{$\bar x_h \in \mathds{R}^m, (\bar a_h, \bar c_h) \in \mathds{R}^{m+1}$}
	\caption{}
\end{algorithm}
$\,$\newline
\noindent We initialize our algorithm with $\bar a_h = 0 , \bar c_h = \left\{ \right\}, \epsilon=10^{-10}$ and solve \eqref{eq:HHNfindimproblem2} using the MATLAB routine 'fmincon' with the following choices: Algorithm: 'active-set'; MaxFunctionEvaluations: $10^5$; MaxIterations: $10^4$; FunctionTolerance: $10^{-12}$.

\subsection{Numerical Examples}
{
As our first example, we consider the setting of \cite[5.3. Example 1]{hafemeyer2019} with known solution, i.e. we set $a=1, d=0$. The following choices satisfy the optimality conditions as stated in Theorem~\ref{thm:HHNoptconddisc}:}
\begin{itemize}
	\item $c := 12 - 4 \sqrt{8}$; $\; x_c := \tfrac{1}{2 \pi} \arccos (\tfrac{c}{4})$; $\; \alpha := 10^{-5}$;
	\item $\bar \ct := 0.5 + 1_{(x_c,1)} - 2 \cdot 1_{(0.5,1)} + 1.5 \cdot 1_{(1-x_c,1)}$;
	\item $\bar \st := \bar \st(\bar \ct)$;
	\item $\bar \Phi (x) := \tfrac{\alpha}{2 c} \left[ (1-\cos (4 \pi x)) - c(1-\cos (2 \pi x)) \right]$;
	\item $\bar \ad := \bar \Phi '$ ;
	\item $\st_{\operatorname{d}} := \bar \st + \bar \ad ''$.
\end{itemize}
In Figure \ref{fig:HHNex1fine} the approximated solutions on a grid with $h = \tfrac{1}{2048}$ are depicted. 

\begin{figure}[ht]
	\begin{center}
		\setlength{\tabcolsep}{1pt}
		\begin{tabular}{ |c c c c|}
			\hline 
			Control & State & Adjoint State & Multiplier \\
			\raisebox{-0\height}{\includegraphics[width=0.24\textwidth]{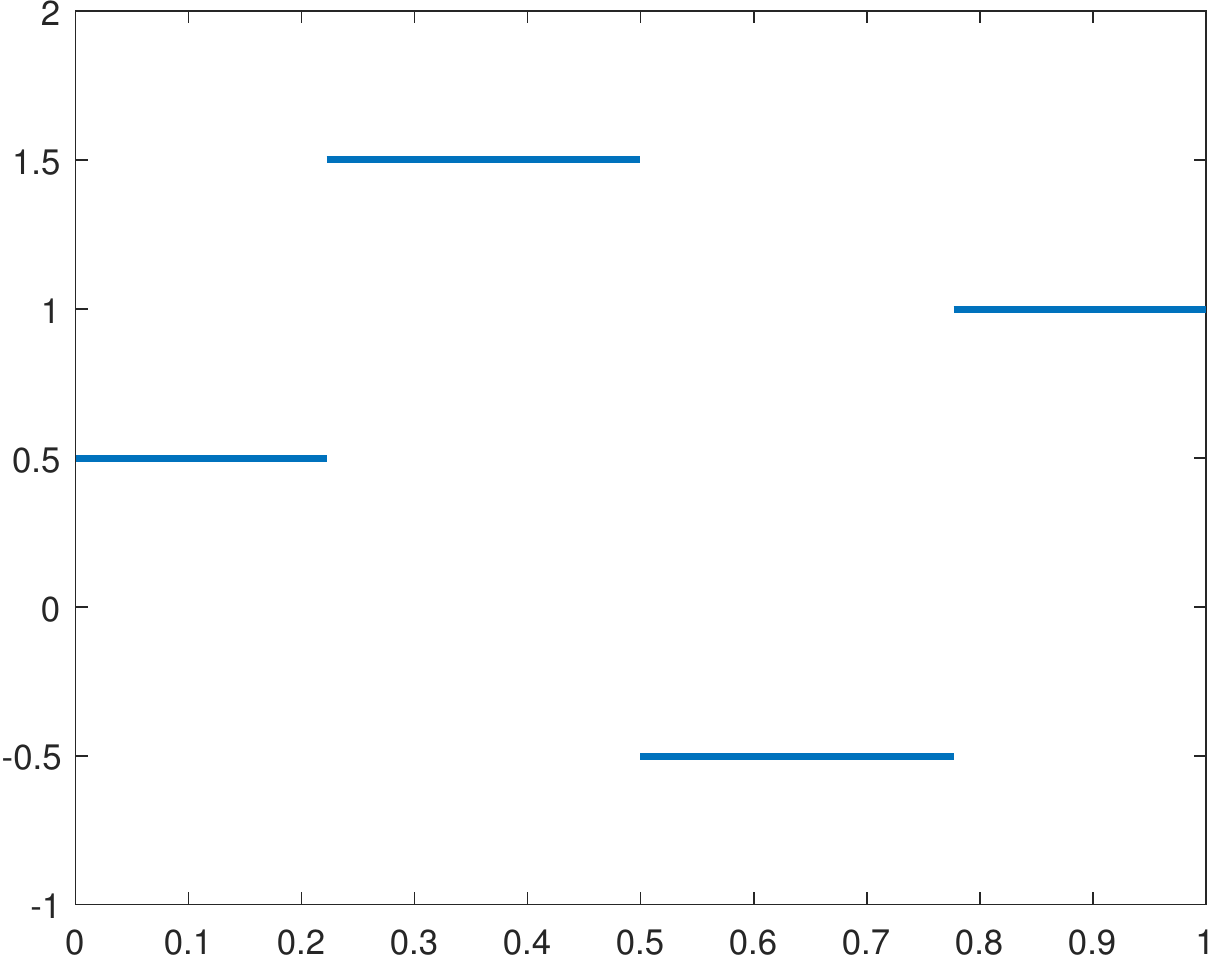}} 
			&
			\raisebox{-0\height}{\includegraphics[width=0.24\textwidth]{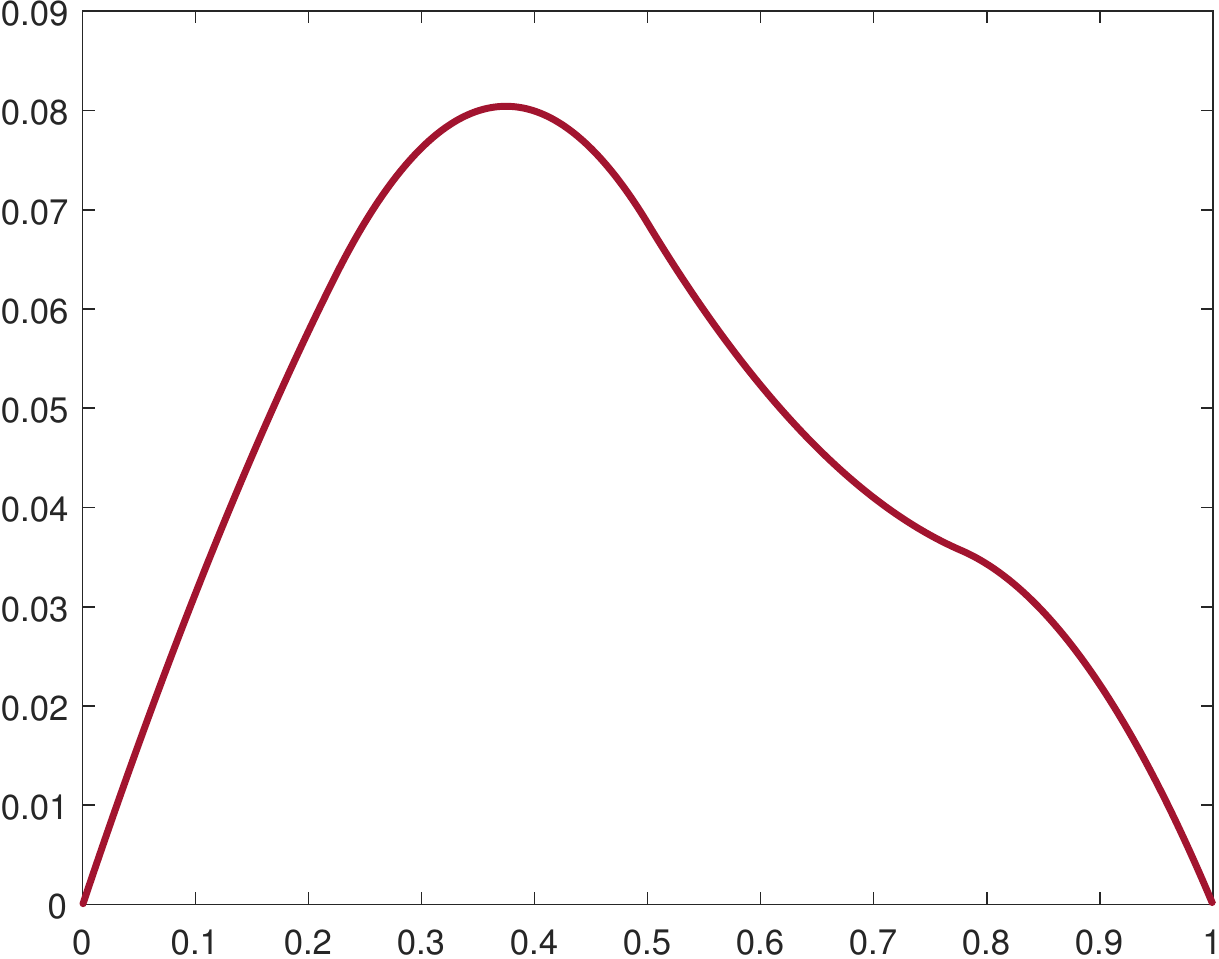}}
			& 
			{\includegraphics[width=0.24\textwidth]{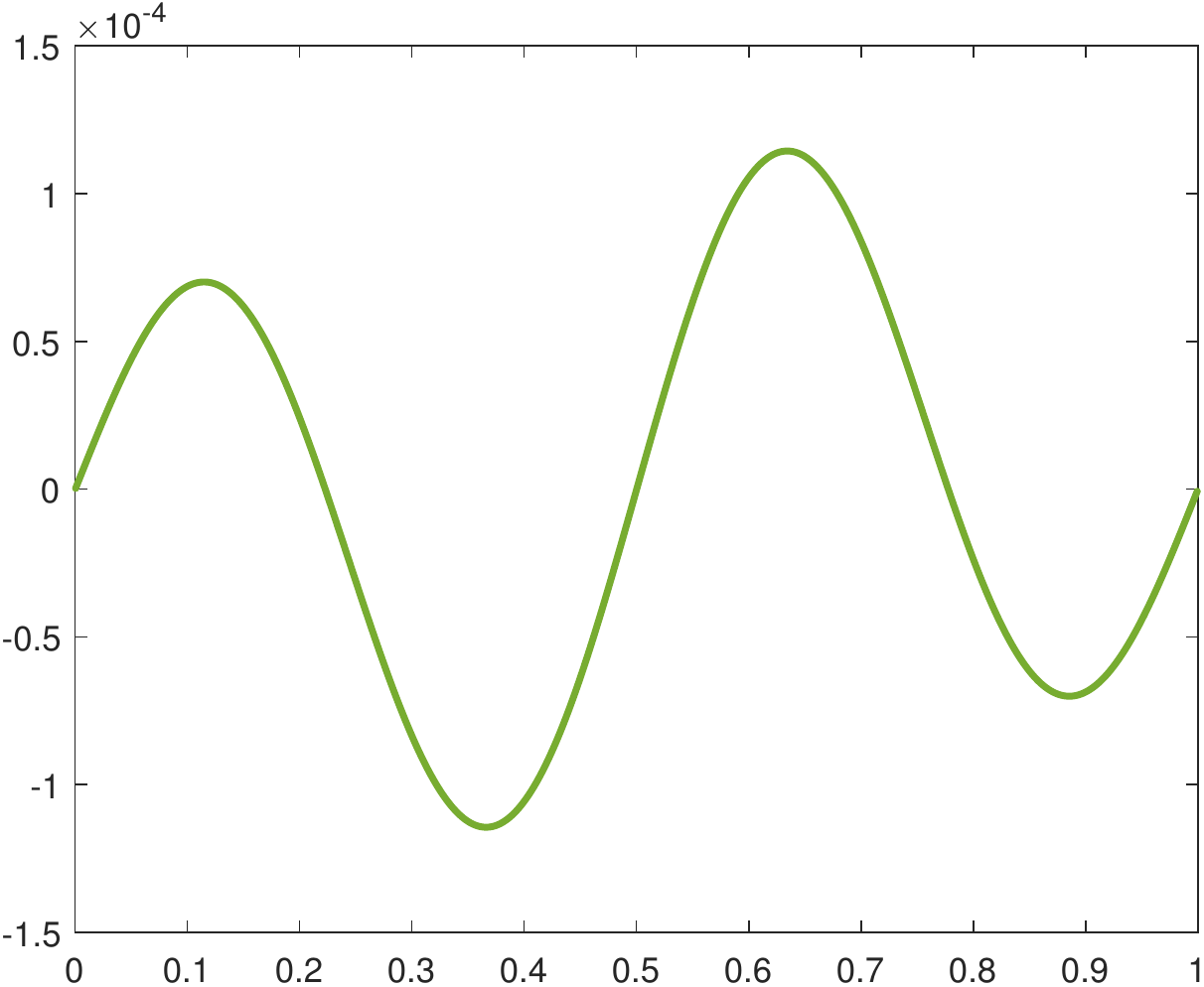}}
			&\raisebox{-0\height}{\includegraphics[width=0.24\textwidth]{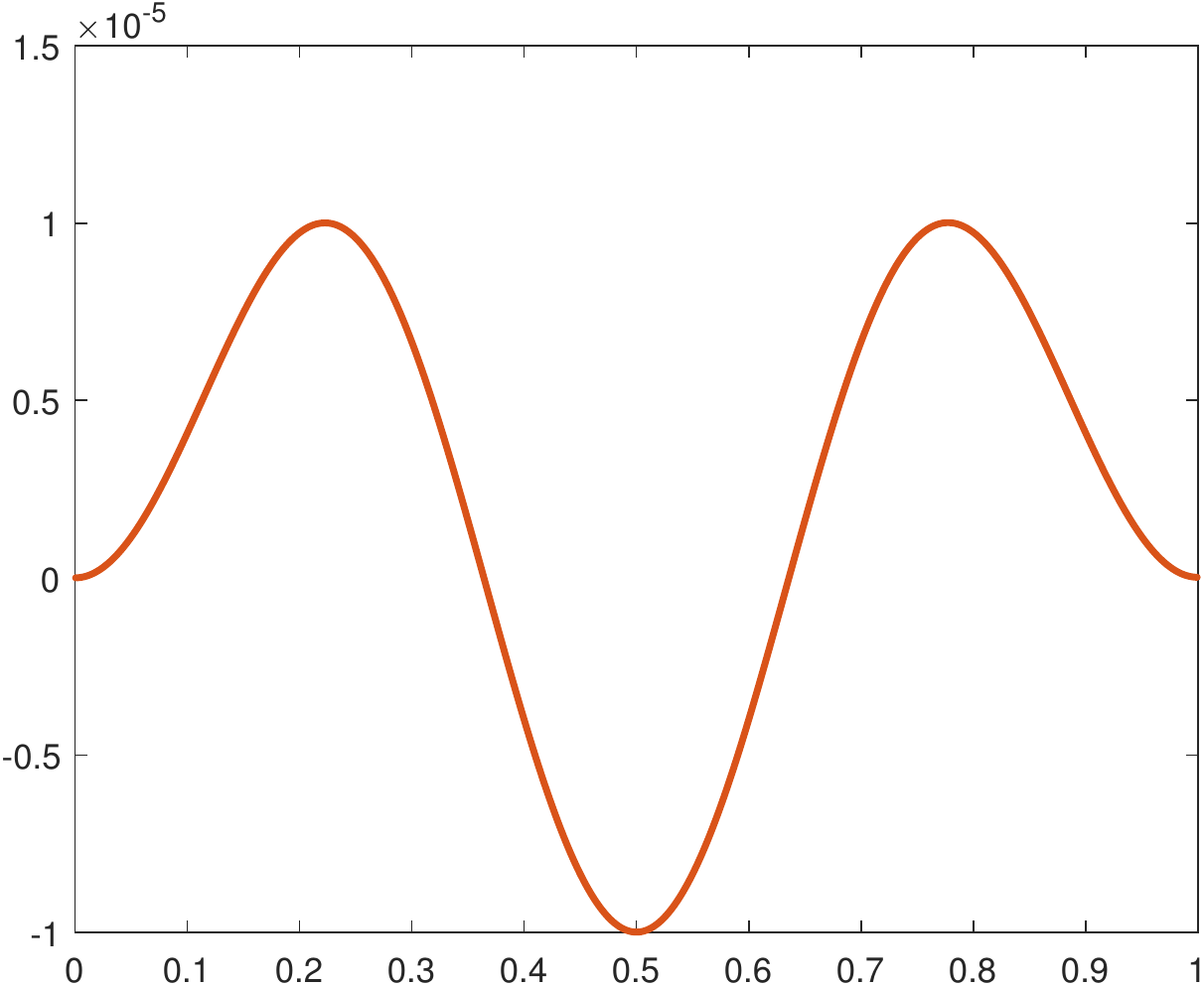}}
			\\
			(A) $\bar \ct_h $ & (B) $\bar \st_h$ & (C) $\bar \ad_h$ & (D) $\bar \Phi_h$ \\
			\hline 
		\end{tabular}
	\end{center}
	\caption{The variationally discrete solution to the data from Example 1 for $h = \tfrac{1}{2048}$. The inclusions in \eqref{eq:HHNsparsedisc3} are clearly visible.}
	\label{fig:HHNex1fine}
\end{figure}

In Figure \ref{fig:HHNex1conv} the errors between the known solutions and the solutions to the variationally discretized problem are displayed. We observe that the order of convergence is approximately $h$, except for $|| \bar \ct - \bar \ct_h||_{L^2(Q)}$, which as expected converges only with the half rate.

\begin{figure}
	\begin{center}
		\includegraphics[width=0.8\textwidth]{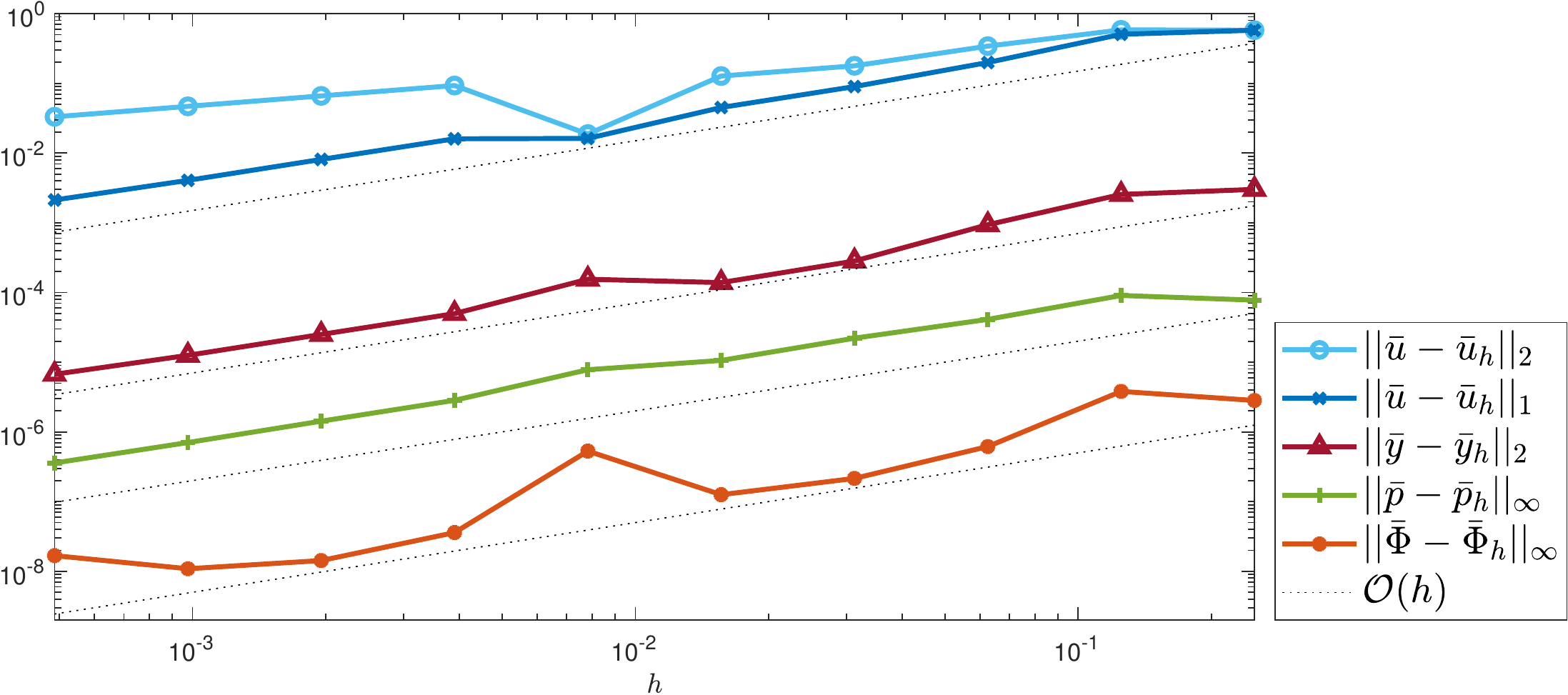}
		\caption{Example 1: Convergence plots of the errors of the solutions to the variationally discrete problem compared to the known exact solution.}
		\label{fig:HHNex1conv}
	\end{center}
\end{figure}
In addition to plotting the errors, we also calculate the convergence order $h^\alpha$ for the refinement from some gridsize $h_1$ to some other gridsize $h_2$, see Table \ref{tab:HHNex1}, by 
\begin{equation*}
	\alpha = \frac{\log(\frac{e_{h_1}}{e_{h_2}})}{\log(\frac{h_1}{h_2})},
\end{equation*}
where $e_{h_1}$ and $e_{h_2}$ act as placeholders for the different errors we are examining, in particular: $|| \bar \ct - \bar \ct_h||_{L^1(Q)}$, $|| \bar \ct - \bar \ct_h||_{L^2(Q)}, || \bar \st - \bar \st_h||_{L^2(Q)}, || \bar \ad - \bar \ad_h||_{L^{\infty}(Q)}$, and $|| \bar \Phi - \bar \Phi_h||_{L^{\infty}(Q)}$.

\begin{table}[ht]
	\begin{tabular}[h]{l|l|l|l|l|l|l} 
		$h_1$ & $h_2$           & $|| \bar \ct - \bar \ct_h||_{L^1}$     & $|| \bar \ct - \bar \ct_h||_{L^2}$     & $|| \bar \st - \bar \st_h||_{L^2}$     & $|| \bar \ad - \bar \ad_h||_{L^{\infty}}$   & $|| \bar \Phi - \Phi_h||_{L^{\infty}}$ \\
		\hline
		\hline
		0.2500 & 0.1250           & 0.1943 & -0.0171 & 0.2403 & -0.2224 & -0.4324 \\
		\hline
		0.1250 & 0.0625         & 1.3436 & 0.7759 & 1.4444 & 1.1389 & 2.6278 \\
		\hline
		0.0625 & 0.0313       & 1.1471 & 0.9368 & 1.7284 & 0.8966 & 1.5183 \\
		\hline
		0.0313 & 0.0156      & 0.9982 & 0.4874 & 1.0286 & 1.0597 & 0.7761 \\
		\hline
		0.0156 & 0.0078   & 1.4732 & 2.7648 & -0.1603 & 0.4420 & -2.0774 \\
		\hline
		0.0078 & 0.0039 & 0.0178 & -2.3127 & 1.6393 & 1.4590 & 3.8838 \\
		\hline
		0.0039 & 0.0020 & 0.9832 & 0.4948 & 0.9920 & 0.9936 & 1.3328 \\ 
		\hline
		0.0020 & 0.0010 & 0.9975 & 0.4975 & 0.9957 & 1.0184 & 0.3887 \\
		\hline
		0.0010 & 0.0005 & 0.9353 & 0.4984 & 0.9025 & 0.9738 & -0.6235 \\
		\hline
		\hline
		&           mean & 0.8989 & 0.4584 & 0.9790 & 0.8622 & 0.8216 \\
		\hline
		& slope of best fit& 0.9307 & 0.4854 & 1.0089 & 0.9241 & 0.9608
	\end{tabular}
	\caption{Example 1: Convergence order (potency of gridsize $h$) of the respective errors when the grid is refined from gridsize $h_1$ to gridsize $h_2$. We remark that the gridsizes are rounded.}
	\label{tab:HHNex1}
\end{table}

\noindent As a second example we use \cite[5.4. Example 2]{hafemeyer2019} with unknown solution, $\alpha = 10^{-5}$ and \linebreak$\st_{\operatorname{d}}(x):= 0.5 \pi^{-2}\left( 1- \cos(2 \pi x)\right) $.

Since the solution is not known, we calculate a reference solution on the finest grid with $h = \tfrac{1}{1024}$. The results displayed in Figure \ref{fig:HHNex2fine} are then used to approximate $\bar \ct, \bar \st, \bar \ad, \bar \Phi$ for the calculation of the errors.

\begin{figure}
	\begin{center}
		\setlength{\tabcolsep}{1pt}
		\begin{tabular}{|c c c c|}
			\hline	
			Control & State & Adjoint state & Multiplier \\ \raisebox{-0\height}{\includegraphics[width=0.24\textwidth]{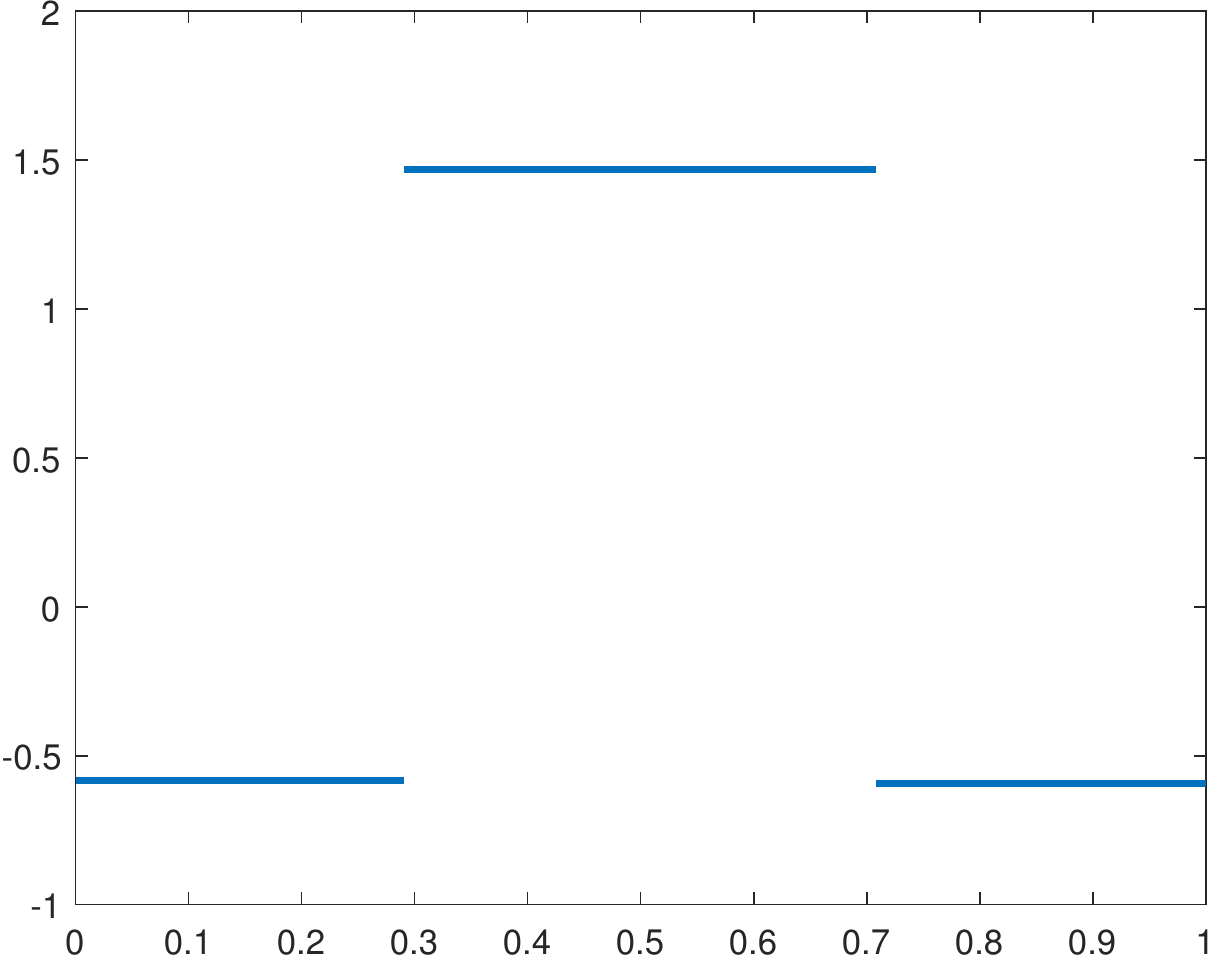}}
			& \raisebox{-0\height}{\includegraphics[width=0.24\textwidth]{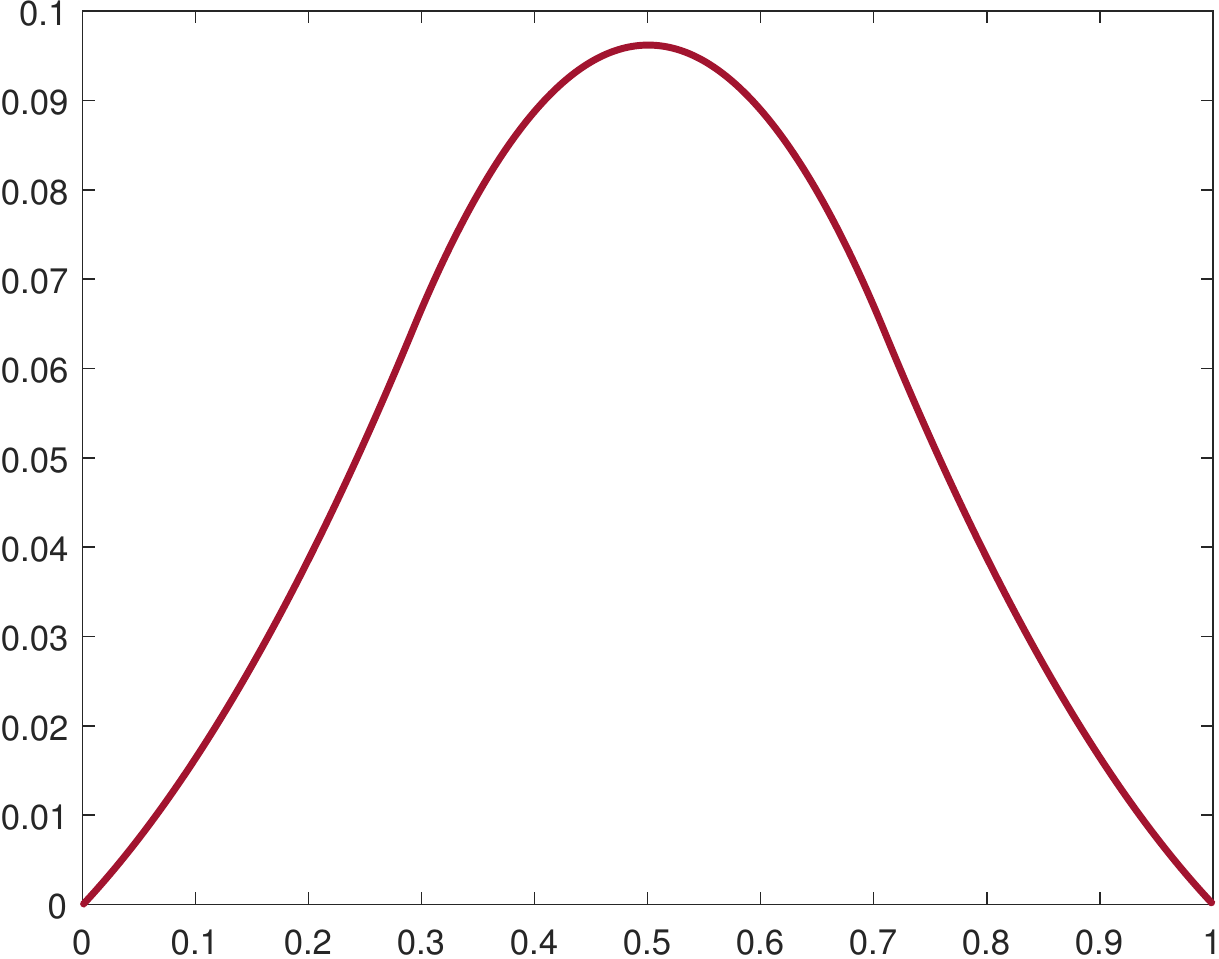}} 
			&
			{\includegraphics[width=0.24\textwidth]{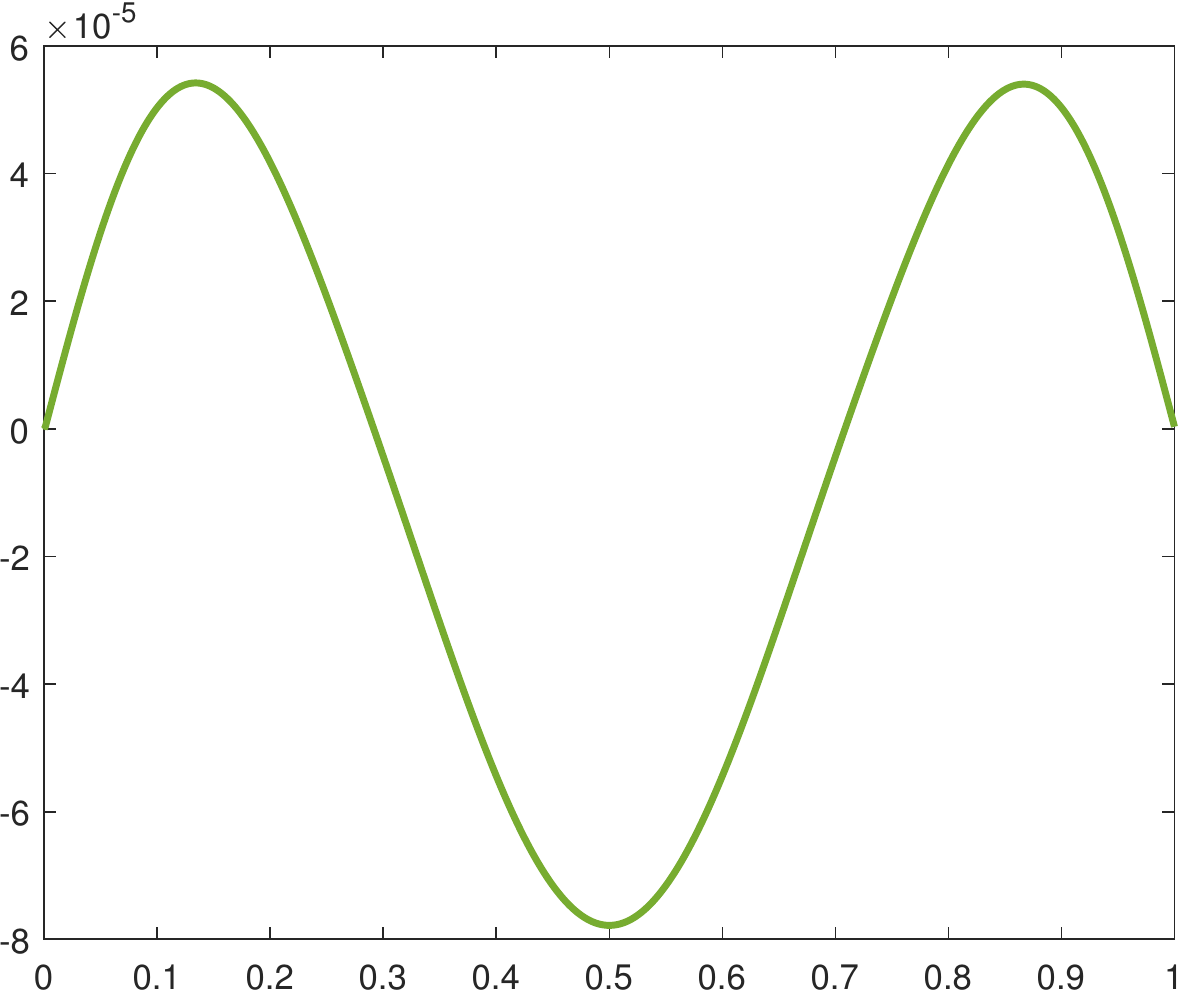}}
			&
			\raisebox{-0\height}{\includegraphics[width=0.24\textwidth]{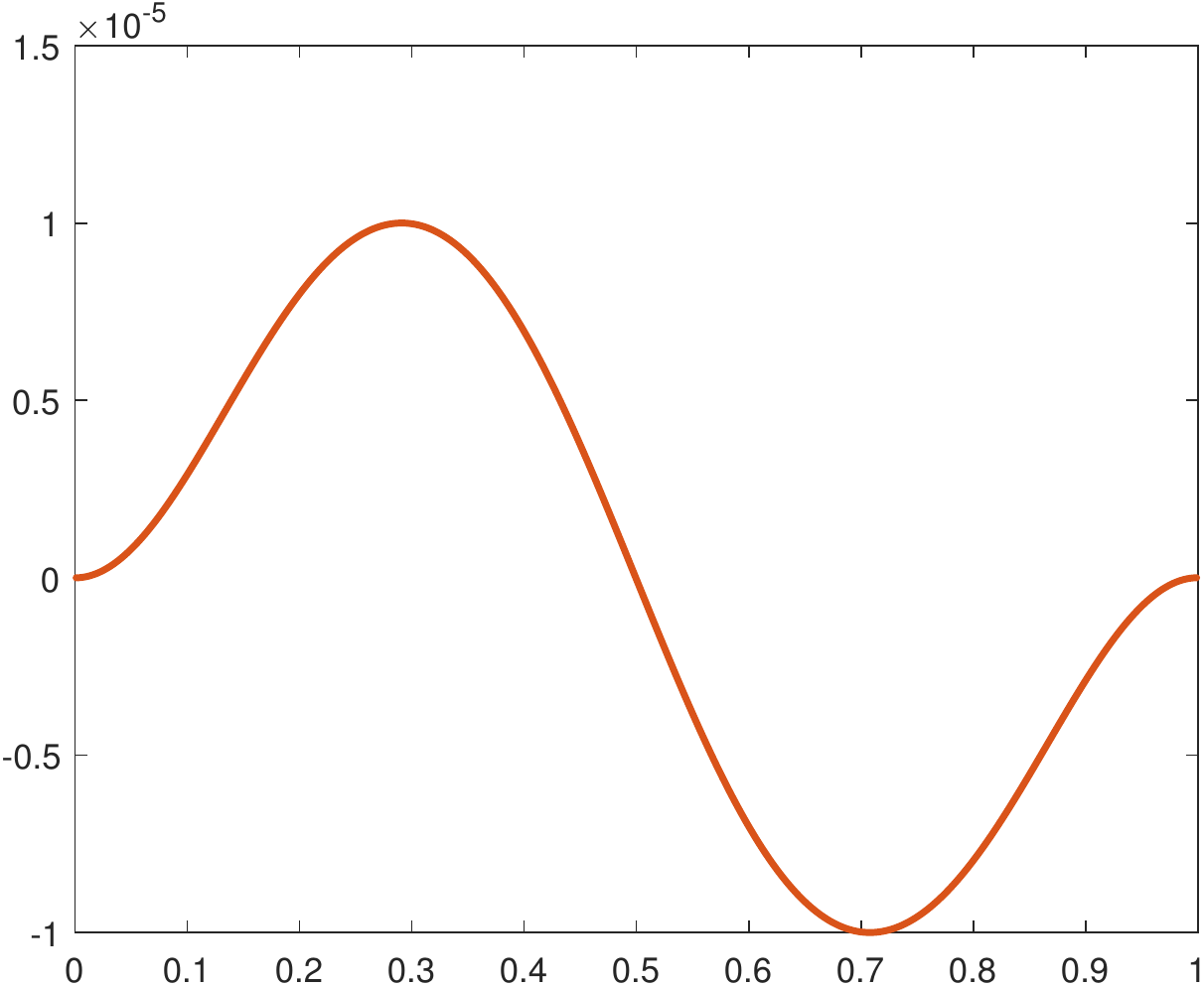}} 
			\\
			(A) $\bar \ct_h$ & (B) $\bar \st_h$ & (C) $\bar \ad_h$ & (D) $\bar \Phi_h$ \\
			\hline
		\end{tabular}
	\end{center}
	\caption{The variationally discrete solution to the data from Example 2 for $h = \tfrac{1}{1024}$. The inclusions in \eqref{eq:HHNsparsedisc3} are clearly visible.}
	\label{fig:HHNex2fine}
\end{figure}

In Figure \ref{fig:HHNex2conv} the errors between the known solutions and the solutions to the variationally discretized problem are depicted. Again, we observe that the order of convergence is approximately $h$, except for $|| \bar \ct - \bar \ct_h||_{L^2(Q)}$, which converges with the half rate. 

\begin{figure}
	\begin{center}
		\includegraphics[width=0.8\textwidth]{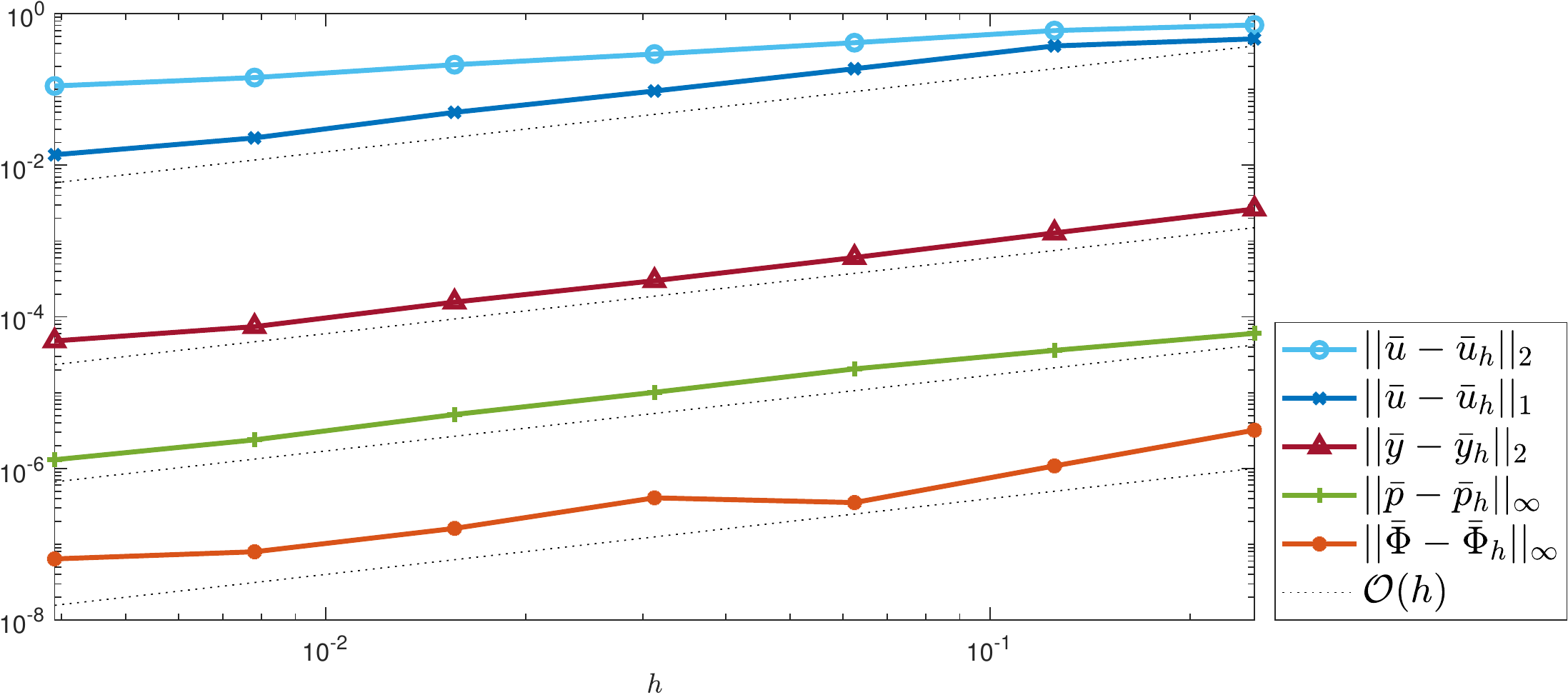}
		\caption{Example 2: Convergence plots of the errors of the solutions to the variationally discrete problem compared to the approximation of the exact solution. The reference solution is computed on a grid with $h = \tfrac{1}{1024}$.}
		\label{fig:HHNex2conv}
	\end{center}
\end{figure}

Furthermore, we calculate the convergence order $h^\alpha$ for the refinement from some gridsize $h_1$ to some other gridsize $h_2$ as explained before. The results are displayed in Table \ref{tab:HHNex2}.

\begin{table}[ht]
	\begin{tabular}[h]{l|l|l|l|l|l|l} 
		$h_1$ & $h_2$           & $|| \bar \ct - \bar \ct_h||_{L^1}$     & $|| \bar \ct - \bar \ct_h||_{L^2}$     & $|| \bar \st - \bar \st_h||_{L^2}$     & $|| \bar \ad -\bar \ad_h||_{L^{\infty}}$   & $|| \bar \Phi - \Phi_h||_{L^{\infty}}$\\
		\hline
		\hline
		0.2500 & 0.1250           & 0.3110 & 0.2454 & 1.0464 & 0.7448 & 1.5684 \\
		\hline
		0.1250 & 0.0625         & 0.9990 & 0.5319 & 1.0788 & 0.8119 & 1.6061 \\
		\hline
		0.0625 & 0.0313       & 0.9763 & 0.4961 & 1.0147 & 1.0266 & -0.2077 \\
		\hline
		0.0313 & 0.0156     & 0.9348 & 0.4682 & 0.9376 & 0.9737 & 1.3400 \\
		\hline
		0.0156 & 0.0078   & 1.1204 & 0.5630 & 1.0757 & 1.1106 & 1.0238 \\
		\hline
		0.0078 & 0.0039 & 0.7379 & 0.3679 & 0.6267 & 0.8702 & 0.3120 \\
		\hline
		\hline
		&  mean          & 0.8466 & 0.4454 & 0.9633 & 0.9230 & 0.9404\\
		\hline
		& slope of best fit & 0.9004 & 0.4679 & 0.9823 & 0.9450 & 0.9137
	\end{tabular}
	\caption{Example 2: Convergence order (potency of gridsize $h$) of the respective errors when the grid is refined from gridsize $h_1$ to gridsize $h_2$. We remark that the gridsizes are rounded.}
	\label{tab:HHNex2}
\end{table}


\noindent Altogether, we are able to verify the results we show in Section~\ref{sec:HHNVD}, i.e.  the inclusions from \eqref{eq:HHNsparsedisc3}, the sparsity structure of the control, and the error estimates for control, state, adjoint state and multiplier. \\

In \cite[Section 5]{hafemeyer2019} the same examples have been analyzed, but without employing a mixed formulation for the state equation. Under almost the same structural assumptions they get the following results: For a variational discretization approach with piecewise linear and continuous state and test functions they observe errors of the order $\mathcal{O}(h^2)$. Additionally, for a full discretization with piecewise constant control and piecewise linear and continuous state and test functions they see errors of the order $\mathcal{O}(h)$. 

In comparison, we consider a variational discretization approach combined with a mixed formulation of the state equation discretized with lowest order Raviart Thomas elements, which corresponds to $(\mix_h,\st_h) \in P_1 \times P_0$. We see that under the given structural assumption this leads to piecewise constant controls without discretizing the control. This clearly demonstrates that the discrete structure of the control in the variational discretization strategy can be \textit{controlled} through the scheme used for the discretization of the state equation.

We note that we obtain the same approximation order for the variationally discrete, piecewise constant controls as \cite{hafemeyer2019} for the full discretization with piecewise constant controls. However, the numerical analysis for the variationally discrete approach in our opinion is much simpler and more natural. We also note that the mixed finite element requires more degrees of freedom than the classical finite element approximation with piecewise linear, continuous elements. This then pays off with a more accurate approximation of the derivative of the state, which might be advantageous in situations where the accurate numerical approximation e.g. of stresses is required.\\

\textbf{Acknowledgment:} We thank Ira Neitzel for fruitful discussions on the topic, in particular related to the work \cite{hafemeyer2019}. We also thank the anonymous referees for their careful reading of the manuscript and the constructive remarks, which helped to improve our manuscript significantly.


\end{document}